\documentclass[noams,lp]{compositio}

\usepackage{amsbsy,amssymb,amscd,amsfonts,latexsym,amstext,delarray,
 amsmath,stackrel,lineno,tabularx,url}

\input xypic
\usepackage[pdfborder={0 0 0}]{hyperref}

\hypersetup{
    colorlinks = true,
    linkcolor = blue,
    anchorcolor = red,
    citecolor = green,
    filecolor = red,
    pagecolor = red,
    urlcolor = magenta
}
\usepackage[filename=tooltipy,movetips,mouseover,noextratext,inactive]{fancytooltips}

\newcommand*\patchAmsMathEnvironmentForLineno[1]{%
  \expandafter\let\csname old#1\expandafter\endcsname\csname #1\endcsname
  \expandafter\let\csname oldend#1\expandafter\endcsname\csname end#1\endcsname
  \renewenvironment{#1}%
     {\linenomath\csname old#1\endcsname}%
     {\csname oldend#1\endcsname\endlinenomath}}%
\newcommand*\patchBothAmsMathEnvironmentsForLineno[1]{%
  \patchAmsMathEnvironmentForLineno{#1}%
  \patchAmsMathEnvironmentForLineno{#1*}}%
\AtBeginDocument{%
\patchBothAmsMathEnvironmentsForLineno{equation}%
\patchBothAmsMathEnvironmentsForLineno{align}%
\patchBothAmsMathEnvironmentsForLineno{flalign}%
\patchBothAmsMathEnvironmentsForLineno{alignat}%
\patchBothAmsMathEnvironmentsForLineno{gather}%
\patchBothAmsMathEnvironmentsForLineno{multline}%
}

\newtheorem{thm}{Theorem}[section] 
\newtheorem{defn}[thm]{Definition} 
\newtheorem{prop}[thm]{Proposition}
\newtheorem{cor}[thm]{Corollary}
\newtheorem{lem}[thm]{Lemma}

\newtheorem{rem}[thm]{Remark}

\def\Gal{{\rm Gal}}

\def\Hom{{\rm Hom}}

\def\Spec{{\rm Spec\,}}
\def\Sp{{\rm Spec}\,}

\def\A{{\mathbb A}}

\def\C{{\mathbb C}}
\def\F{{\mathbb F}}

\def\N{{\mathbb N}}
\def\P{{\mathbb P}}
\def\Q{{\mathbb Q}}
\def\R{{\mathbb R}}
\def\Z{{\mathbb Z}}

\def\cA{{\mathcal A}}
\def\cB{{\mathcal B}}
\def\cC{{\mathcal C}}
\def\cD{{\mathcal D}}

\def\cL{{\mathcal L}}

\def\cM{{\mathcal M}}
\def\cO{{\mathcal O}}

\def\cT{{\mathcal T}}
\def\cU{{\mathcal U}}
\def\cV{{\mathcal V}}

\def\qqq{\,,\quad~\forall}

\newcommand{\ie}{{\it i.e.\/}\ }
\newcommand{\eg}{{\it e.g.\/}\ }
\newcommand{\cf}{{\it cf.\/}\ }
\newcommand{\opcit}{{\it op.cit.\/}\ }
\newcommand{\resp}{{\it resp.\/}\ }

\newcommand{\har}{HC^{\rm ar}}

\def\dim{{\mbox{dim}}}

\def\Hom {{\mbox{Hom}}}

\newcommand{\aaa}{{\rm alg}}
\newcommand{\ann}{{\rm an}}
\newcommand{\sss}{{\rm sm}}
\newcommand{\rrr}{{{\rm real},\lambda}}

\def\ffp{\mathfrak{p}}
\def\ffq{\mathfrak{q}}
\def\ffm{\mathfrak{m}}

\newcommand{\lam}{\Lambda}

\newcommand{\aff}{{\rm aff}}

\renewcommand{\projlim}{\varprojlim}

\newcommand{\tot}{{\rm Tot}}

\begin{document}

\title{Cyclic homology, Serre's local factors and the $\lambda$-operations}
\author{Alain Connes}
\email{alain@connes.org}
\address{Coll\`ege de France,
3 rue d'Ulm, Paris F-75005 France\newline
I.H.E.S. and Ohio State University.}
\author{Caterina Consani}
\email{kc@math.jhu.edu}
\address{Department of Mathematics, The Johns Hopkins
University\newline Baltimore, MD 21218 USA.}
%
\thanks{The second author is partially supported by the NSF grant DMS 1069218
and would like to thank the Coll\`ege de France for some financial support. The first author thanks A. Abbes for discussions on the Appendix.
}

\begin{abstract}

 We show that for a smooth, projective variety $X$ defined over a number field $K$, cyclic homology with coefficients in the ring $\A_\infty=\prod_{\nu|\infty} K_\nu$,
provides the right theory to obtain, using the $\lambda$-operations,  Serre's
archimedean local factors of the complex L-function of $X$ as regularized determinants.
\end{abstract}

\maketitle

\tableofcontents  

\section{Introduction}
Cyclic cohomology was introduced and widely publicized by the first author of this paper in 1981 (\cf\cite{ac}) as an essential tool in noncommutative differential geometry (\cf\cite{ncg}, \cite{NCG}). In the context of algebraic geometry the dual theory, cyclic homology, which had been anticipated in the commutative case in \cite{rein}, was subsequently developed in a purely algebraic framework  by J.L. Loday--D. Quillen (\cf \cite{LQ}, \cite{Loday}),  B. Feigin--B. Tsygan (\cf \cite{Tsygan}, \cite{FT}), M. Karoubi (\cf\cite{Karoubi}), C.Hood--J. Jones (\cf\cite{HoodJones}) and in the more recent work of C. Weibel (\cf\cite{weibel}, \cite{Weibelcris}) to which we refer as a basic reference throughout this paper.\newline
In this article we show that for a smooth, projective variety $X$ defined over a number field $K$, cyclic homology with coefficients in the ring $\A_\infty=\prod_{\nu|\infty} K_\nu$,
provides the right theory to obtain, using the $\lambda$-operations,  Serre's
archimedean local factors of the complex L-function of $X$ as regularized determinants. In \cite{Deninger91}, C. Deninger made an attempt at constructing a cohomological theory that achieves a similar goal by introducing an archimedean (yet simplified) analogue of J.M. Fontaine's $p$-ring ${\bf B}_{{\rm dR}}$. It is evident that due to the infinite number of poles of the Gamma functions which enter in Serre's formula of  the archimedean local factors
one is required to look for an infinite dimensional cohomological theory (\cf\cite{Katiathesis}). Moreover, since the multiplicity of these poles is provided by the dimension (in a precise range) of the real Deligne cohomology of the complex \resp real variety $X_\nu=X\times_K K_\nu$ over $K_\nu$,  one regards this latter as a cohomological theory intimately related to the sought for ``archimedean cohomology''. As pointed out in \cite{Deninger91} there is however a mismatch happening at a {\em real} place $\nu$ and for odd Hodge weights $w$, between Deninger's proposed definition of the archimedean cohomology  and the real Deligne cohomology of $X_\nu=X_{/\R}$.
 In this paper we show that for $\nu\vert\infty$ the cyclic homology of $X_\nu$ provides a conceptual general construction of the sought for archimedean cohomology of $X$ in terms of what we define (\cf Definition \ref{realhar}) as {\em archimedean cyclic homology} $\har_*(X_\nu)$.\newline
 Cyclic homology is naturally  infinite dimensional, it solves the above mismatch with the real Deligne cohomology and at the same time also unveils the subtle nature of the operator $\Theta$, as generator of the $\lambda$-operations, whose regularized determinant yields the archimedean local factors.
  The endomorphism $\Theta$ has two constituents: the natural grading in cyclic homology and the action of the multiplicative semigroup
$\N^\times$ on cyclic homology of {\em commutative} algebras given by the $\lambda$-operations $\Lambda(k)$, $k\in \N^\times$. More precisely, the action $u^\Theta$ of the multiplicative group $\R_+^\times$ generated by $\Theta$ on cyclic homology, is uniquely determined by its restriction to the dense subgroup $\Q_+^\times\subset \R_+^\times$ where it is  given by the formula
\begin{equation}\label{actiontheta}
    k^\Theta|_{HC_n(X_\nu)}=\Lambda(k)\,k^{-n} \qqq n\geq 0, \,  \ k\in \N^\times\subset \R_+^\times.
\end{equation}
Our main result is the following
 \begin{thm}\label{MT} Let $X$ be a smooth, projective variety of dimension $d$ over an algebraic number field $K$ and let $\nu\vert\infty$ be an archimedean place of $K$.
Then, the action of the operator $\Theta$ on the archimedean cyclic homology of $X_\nu$ satisfies the following formula
\begin{equation}\label{dettheta0}
    \prod_{0\leq w \leq 2d} L_\nu(H^w(X),s)^{(-1)^{w+1}}=\frac{det_\infty(\frac{1}{2\pi}(s-\Theta))|_{\har_{\rm even}(X_\nu)}}{
    det_\infty(\frac{1}{2\pi}(s-\Theta))|_{\har_{\rm odd}(X_\nu)}},\qquad s\in\R.
\end{equation}
The left-hand side of \eqref{dettheta0} is the product of Serre's archimedean local factors of the complex $L$-function of $X$ (\cf\cite{Se3}). On the right-hand side, $det_\infty$ denotes the regularized determinant (\cf\eg\cite{RS},\cite{Deninger91}) and one sets $\har_{\rm even}(X_\nu)=\bigoplus_{n=2k\ge 0} \har_{n}(X_\nu)$, $\har_{\rm odd}(X_\nu)=\bigoplus_{n=2k+1\ge 1} \har_{n}(X_\nu)$.
\end{thm}
 By taking into account the fact that cyclic homology of a finite product of algebras (or disjoint union of schemes) is the direct sum of their cyclic homologies, \eqref{dettheta0} determines the required formula for the product of all archimedean local factors in terms of cyclic homology with coefficients in the ring $\A_\infty=\prod_{\nu|\infty} K_\nu$.

Let us now describe in some details the archimedean cyclic homology groups $\har_*(X_\nu)$ appearing on the right-side of \eqref{dettheta0}.
The nuance between the archimedean cyclic homology $\har_*(X_\nu)$  and ordinary cyclic homology $HC_*(X_\nu)$ as developed in the context of algebraic geometry, corresponds exactly to the difference between the real Deligne cohomology and reduced Deligne cohomology of $X_\nu$ \ie relative de Rham cohomology of $X_\nu$ (up to a shift of degrees). The cyclic homology groups $HC_*(X_\nu)$ have coefficients in $\C$ for a complex place $\nu$ of $K$, \resp in $\R$ for a real place, and in this context they play the role of relative de Rham cohomology. The {\em real} Betti cohomology  of a complex projective variety $X_\C$ is recovered in periodic cyclic homology theory
by the inclusion (\cf \S\ref{pch}, Proposition~\ref{lemperbis})
\begin{equation}\label{inclus}
HP_*(C^\infty(X_\sss,\R))\subset HP_*(C^\infty(X_\sss,\C))\cong HP_*(X_\C),
\end{equation}
where $X_\sss$ denotes the underlying smooth $C^\infty$ manifold. The periodic cyclic homology groups $HP_*(C^\infty(X_\sss,\R))$ and $HP_*(C^\infty(X_\sss,\C))$ are defined in topological terms in \cite{ncg}. Then the {\em real ($\lambda$)-twisted periodic cyclic homology} $HP_*^{\rrr}(X_\C)$ is defined as the hyper-cohomology of a complex of cochains (\cf Definition \ref{maptau}) described by the quasi-pullback of two  natural maps of complexes in periodic cyclic theory. The resulting map $\tau:HP_*^{\rrr}(X_\C)\to HP_*(X_\C)$ is essentially the  Tate twist $(2\pi i)^{\Theta_0}$ combined with the above inclusion \eqref{inclus}. Here, $\Theta_0$ denotes the generator of the $\lambda$-operations, \ie it differs from the above operator $\Theta$ by the grading in cyclic homology. Let us first consider the case when the  archimedean place  $\nu$   of $K$ is complex. Then, one obtains (\cf Proposition \ref{agree}) the short exact sequence (which determines the archimedean cyclic homology)
\begin{equation}\label{defnC}
   0\to HP_{*+2}^\rrr(X_\nu)\stackrel{S\circ \tau}{\to}  HC_*(X_\nu)\to \har_*(X_\nu)\to 0
\end{equation}
where the periodicity map $S$ implements here the link between  periodic cyclic homology and cyclic homology. If instead the archimedean place $\nu$ is real, one takes in the above construction the fixed points
  of the anti-linear conjugate Frobenius operator $\bar F_\infty$ acting on $HC_*(X_\nu\otimes\C)$ (\cf \S \ref{sectreal}). While the  exact sequence \eqref{defnC} suffices to define $\har_*(X_\nu)$, we also provide in Definition~\ref{realhar} the general construction of  a complex of cochains whose cohomology gives $\har_*(X_\nu)$. This complex is defined by the quasi-pullback of  two maps connecting the negative cyclic complex (\cf \cite{HoodJones}) \resp the real ($\lambda$)-twisted periodic  cyclic complex to the periodic cyclic complex.

Besides its natural clarity  at the conceptual level, this newly developed archimedean cyclic homology theory  has also the following qualities:

$1)$~It inherits by construction the rich structure of cyclic homology, such as the action of the periodicity operator $S$ (that has to be Tate twisted in the complex case and squared in the real case).\newline
$2)$~It is directly connected to algebraic $K$-theory and the $\gamma$-filtration by the regulator maps, thus it acquires naturally a role in the theory of motives.\newline
$3)$~The fundamental formula derived from \eqref{dettheta0} that gives the product of all archimedean local factors in terms of cyclic homology with coefficients in the ring $\A_\infty=\prod_{\nu|\infty} K_\nu$, is clearly of adelic nature and evidently suggests the study of a generalization of these results, by implementing the full ring $\A_K$ of ad\`eles as coefficients. Cyclic homology with coefficients in the number field $K$ should provide a natural lattice in the spirit of \cite{BK}.\newline
$4)$~When translated in terms of the logarithmic derivatives of the two sides, formula \eqref{dettheta0} is very suggestive of the existence of a global Lefschetz formula in cyclic homology.\newline
$5)$~At the conceptual level, cyclic homology is best understood as a way to embed the category of non-commutative $k$-algebras in the abelian category of $\Lambda$-modules (\cf\cite{CoExt}), where $\Lambda$ is a small category built from simplicial sets and cyclic groups. Any non-commutative algebra $\cA$ gives rise {\em canonically}, through its tensor powers $\cA^{\otimes^n}$, to a $\Lambda$-module $\cA^\natural$. The functor
$\cA\mapsto \cA^\natural$ from the category of algebras to the abelian category of $\Lambda$-modules retains all the  information needed to compute the cyclic homology groups. In fact, since these groups are computed as ${\rm Tor}(k^\natural,\cA^\natural)$ the cyclic theory fits perfectly with the use of extensions in the theory of motives.\newline
$6)$~One knows (\cf\cite{CoExt}) that the classifying space of the small category $\Lambda$ is $\P^\infty(\C)$ and its cohomology accounts for the geometric meaning of the periodicity operator $S$. The immersion of the category of algebras in the abelian category of $\Lambda$-modules is refined {\em in the commutative case} by the presence of the $\lambda$-operations. As shown in \S 6.4 of \cite{Loday}, the presence of the $\lambda$-operations for a $\Lambda$-module $E:\Lambda^{\rm op}\to (k-{\rm Mod})$ is a consequence of the fact that $E$ factorizes through the category ${\rm \bf Fin}$ of finite sets. This clearly happens for $E=\cA^\natural$, when the algebra $\cA$ is commutative. We expect that a deeper understanding in the framework of algebraic topology, of the relations between the cyclic category $\Lambda$ and the category ${\rm \bf Fin}$ of finite sets should shed light on the structure of the ``absolute point" $\Spec(\F_1)$ and explain the role of cyclic homology and the $\lambda$-operations in the world of motives.\newline
The paper is organized as follows\newline
In Section \ref{sectreduced}, we recall the definition of the real Deligne cohomology and review briefly the result of \cite{Weibelcris} (\cf also \cite{FT}) which relate the Hodge filtration of the Betti cohomology of a smooth, projective variety over $\C$ to the cyclic homology of the associated scheme. This provides a direct link between the reduced Deligne cohomology (in the sense of \cite{Loday}) and the cyclic homology of the variety.\newline
In Section \ref{sectarch}, we first recall the well known formula describing the multiplicity of the poles of the archimedean local factors of $X$ as the rank of some real Deligne cohomology groups. This result leads unambiguously to the definition of the archimedean cohomology of a smooth, projective algebraic variety $X$ over a number field as an infinite direct sum of real Deligne cohomology groups. Then, we point out that by neglecting at first the nuance between the real Deligne cohomology and the reduced Deligne cohomology, this infinite direct sum is nothing but the cyclic homology group $ \oplus_{n\geq 0} HC_n(X_\C)$. \newline
The remaining sections are then dedicated to describe in cyclic terms  the difference between reduced and unreduced Deligne cohomology.\newline
In Section \ref{sectrealcyc} we use the computation of the cyclic homology of a smooth manifold (\cf \cite{ncg}, Theorem 46) and the stability property of the periodic cyclic homology when passing from the algebraic cyclic homology of a smooth complex projective variety $X_\C$ to the associated $C^\infty$-manifold $X_\sss$ to construct the real ($\lambda$)-twisted periodic cyclic homology $HP^\rrr_*(X_\C)$.\newline
The scheme theoretic relation between $X_\sss$ and $X_\C$ is explored in depth in Appendix \ref{sectweil} where we construct, by elaborating on the functor Weil restriction ${\rm Res}_{\C/\R}$, an additive functor $X\mapsto {}_*X$ which properly belongs to algebraic geometry over $\C$ (\ie maps schemes over $\C$ to schemes over $\C$). This allows one to extend the obvious map of locally ringed spaces $X_\sss\to X_\C$, to a {\em morphism of schemes} $\pi_X:\Spec(C^\infty(X_\sss,\C))\to X_\C$ that factors naturally through the complex scheme $({}_*X)_\C = {}_*X\times\C$
and to which the general theory of \cite{weibel} can be applied.\newline
In Section \ref{sectharcyc} we provide the definition and the first properties of the archimedean cyclic homology theory. In \S \ref{sectreal} we provide the cyclic homology meaning of  the anti-linear conjugate Frobenius operator $\bar F_\infty$ and develop the construction at the real places.\newline
Finally, Section \ref{sectmain} contains a detailed proof of Theorem \ref{MT}.

\section{The reduced Deligne cohomology and cyclic homology}\label{sectreduced}

In this section we shortly review the result of \cite{Weibelcris} from which our paper develops. This basic statement relates the Hodge filtration on the Betti cohomology of a smooth, projective variety over $\C$ to the cyclic homology of the associated scheme. Throughout this paper we follow the conventions of \cite{weibelbook} to denote the shift of the indices in chain $C_*$ \resp cochain complexes $C^*$ \ie
\begin{equation*}\label{shift}
    C[p]_n:=C_{n+p},\ \ C[p]^n:=C^{n-p}.
\end{equation*}
Note that this convention is the opposite of the one used in \cite{Schneider}.

\subsection{Reduced Deligne cohomology}\label{redD}

Let $X_{\C}$ be  a smooth, projective variety over the complex numbers.
One denotes by $\R(r)$ the subgroup $(2\pi i)^r\R$ of $\C$ ($i=\sqrt{-1}$) and by $\R(r)_\mathcal D$ the complex of sheaves of holomorphic differential forms $\Omega^\cdot=\Omega^\cdot_{X(\C)}$ on the complex manifold $X(\C)$ associated to $X_\C$, whose cohomology defines the real Deligne cohomology of $X_\C$:
\[
\R(r)_\mathcal D : ~\R(r)\stackrel{\epsilon}{\to}\Omega^0\stackrel{d}{\to}
\Omega^1_{}\stackrel{d}{\to}\Omega^1_{}\stackrel{d}{\to}\cdots\stackrel{d}{\to}\Omega^{r-1}_{}\to 0.
\]
Here $\R(r)$ is placed in degree $0$ and $\Omega^p_{}$ in degree $p+1$ ($\forall p\ge 0$) and the map $\epsilon$ is the inclusion of the twisted constants: $\R(r)\subset\C\subset\cO_{X(\C)}=\Omega^0_{X(\C)}$. The real Deligne cohomology of $X_{\C}$ is defined as  the hyper-cohomology  of the above complex: $H^n_{\mathcal D}(X_{\C},\R(r)) := \mathbb H^n(X(\C),\R(r)_\mathcal D)$. One has an evident short exact sequence of complexes
\begin{equation}\label{deligne0}
0\to (\Omega^{<r}_{X(\C)})[1]\to \R(r)_\mathcal D\to \R(r)\to 0.
\end{equation}
Following \cite{Loday} (\S 3.6.4), we call the {\em reduced Deligne complex} the kernel of the surjective map
$\R(r)_\mathcal D\to \R(r)$ in \eqref{deligne0}, that is the truncated de Rham complex shifted by $1$ to the right. The {\em reduced Deligne cohomology} is then defined as the hyper-cohomology of the reduced Deligne complex:
\begin{equation}\label{delignered}
\tilde H^n_{\mathcal D}(X_\C,\R(r)): =\mathbb H^n(X(\C),(\Omega^{<r}_{X(\C)})[1])= \mathbb H^n(\text{Cone}(F^r\Omega^\cdot_{X(\C)} \stackrel{\iota}{\to}\Omega^\cdot_{X(\C)})[1]).
\end{equation}
The second equality in \eqref{delignered} is an immediate consequence of the definition of the Hodge sub-complex $F^r\Omega^\cdot_{X(\C)}\stackrel{\iota}{\subseteq}\Omega^\cdot_{X(\C)}$. Up to a shift by $1$, the reduced  Deligne cohomology is sometimes referred to as relative de Rham cohomology. The results of \cite{GAGA} and \cite{Grothendieck} show that the reduced Deligne cohomology is computed in the same way in terms of the de Rham complex of algebraic differential forms on the scheme $X_\C$
and thus makes sense for schemes over $\C$.\newline
In Proposition \ref{propcyclic1} we shall review a generalization to the projective case of Proposition 3.6.5 of \cite{Loday} which relates, for a commutative and smooth $\C$-algebra $A$, the cyclic homology of $A$ with the reduced Deligne cohomology of the affine algebraic spectrum $\Spec(A)$. The algebraic analogue (\cf \cite{Loday}, Theorem 3.4.12) of the computation of cyclic cohomology of smooth manifolds of \cite{ncg} (Theorem 46), implies Proposition 3.6.5 of \cite{Loday}, \ie the isomorphism
\begin{equation*}
    HC_n(A)\cong \bigoplus_{i\geq 0}\tilde H^{n+1-2i}_{\mathcal D}(\Spec(A),\R(n+1-i)).
\end{equation*}
In the affine case, the reduced Deligne cohomology is computed in terms of global K\"{a}hler differentials.

\subsection{$\lambda$-decomposition}
 We recall from \cite{FT}, \cite{Loday} and \cite{Weibelcris}, that one has a natural action of the multiplicative semigroup $\N^\times$ of positive integers on the cyclic homology groups of a {\em commutative} algebra (over a ground ring). This action is the counterpart in cyclic homology of Quillen's  $\lambda$-operations in algebraic $K$-theory (\cf \cite{Loday} 4.5.16)
\begin{prop}\label{eulerdec}
Let $A$ be a commutative algebra and $(C_k(A)=A^{\otimes (k+1)},b,B)$ its mixed complex of chains. The $\lambda$-operations define (degree zero) endomorphisms $\lam_*$ of $C_*(A)$ commuting with the grading and satisfying the following properties\newline
-~ $\lam_{nm}=\lam_n\lam_m$,\quad $\forall n,m\in \N^\times$\newline
-~ $b\lam_m=\lam_m b$,\quad $\forall m\in \N^\times$\newline
-~$\lam_m B=mB\lam_m$,\quad $\forall m\in \N^\times$.
 \end{prop}
 \proof The statements are all proven in  \cite{Loday}. The first is equation (4.5.4.2). The second is Proposition 4.5.9 and the third is  Theorem 4.6.6 of \opcit. \endproof

This construction gives rise, for a commutative algebra $A$ over a ground field of characteristic $0$, to a {\em canonical} decomposition, called the $\lambda$-decomposition, of the cyclic homology of $A$ as a direct sum
  \begin{equation*}\label{lambdadec}
    HC_n(A)=\bigoplus_{j\ge 0} HC_n^{(j)}(A)
  \end{equation*}
  which is uniquely determined as a diagonalization of the endomorphisms $\lam_m$ \ie
   \begin{equation*}\label{spec}
        \lam_m(\alpha)=m^j\,\alpha\qqq \alpha \in HC_n^{(j)}(A),\  m\in \N^\times.
      \end{equation*}

\subsection{Cyclic homology of smooth, projective varieties}

In \cite{Weibelcris} (Lemma 3.0) it is proven that the $\lambda$-decomposition extends to the projective case, thus for a smooth, projective algebraic variety $X_\C$ over the complex numbers one has
the finite decomposition
   \begin{equation*}\label{hcn}
    HC_n(X_\C)=\bigoplus_{j\geq 0} HC_n^{(j)}(X_\C).
   \end{equation*}
We recall the following key result of \cite{Weibelcris} (Theorem 3.3)
  \begin{prop}\label{propcyclic1}
Let  $X_\C$ be a  smooth, projective algebraic variety over $\C$.
Then one has canonical isomorphisms
\begin{equation}\label{cyclicequ1}
    HC_n^{(j)}(X_\C)
    \cong\tilde H^{2j+1-n}_{\mathcal D}(X_\C,\R(j+1))\cong H_B^{2j-n}(X(\C),\C)/F^{j+1}\qquad\forall j\geq 0,\forall n\geq 0.
\end{equation}
\end{prop}
In \cite{Weibelcris} (Theorem 3.3) it is proven that
 \begin{equation*}\label{weib}
   HC_n^{(j)}(X_\C)= \mathbb H^{2j-n}(X_\C,\Omega^{\leq j}_{X_\C})=\mathbb H^{2j-n}(X(\C),\Omega^{\leq j}_{X(\C)}).
 \end{equation*}
Moreover it follows from the degeneration of the ``Hodge to de Rham'' hyper-cohomology spectral sequence  and the canonical identification of de Rham cohomology with the singular cohomology that $\mathbb H^{2j-n}(X(\C),\Omega^{\leq j}_{X(\C)})\cong H_B^{2j-n}(X(\C),\C)/F^{j+1}$.

\section{Cyclic homology and archimedean cohomology}\label{sectarch}

 In this section we first recall the basic result which expresses the order of the poles of Serre's archimedean local factors $L_\nu(H^w(X),s)$ ($\nu\vert\infty$) of the complex L-function of a smooth, projective algebraic variety $X$ over a number field in terms of the ranks of some real Deligne cohomology groups of $X_\nu: = X\times_K K_\nu$ over $K_\nu$. The outcome leads unambiguously to consider an infinite direct sum of real Deligne cohomology groups of $X_\nu$ as the most natural candidate for the archimedean cohomology of $X$. If one ignores at first the nuance between real Deligne cohomology and reduced Deligne cohomology, this infinite direct sum (when $\nu$ is a complex place of $K$) is nothing but the cyclic homology direct sum
  $ \oplus_{n\geq 0} HC_n(X_\C)$. In the last part of the section we shall describe the strategy on how to understand  in full this difference.

\subsection{Deligne cohomology and poles of the archimedean factors}

Let $X$ be  a  smooth, projective variety over a number field $K$. We fix  an archimedean  place $\nu\vert\infty$ of $K$ and a Hodge weight $w$ of the singular cohomology of the complex manifold $X_\nu(\C)$ and consider the local factor $L_\nu(H^w(X),s)$. In view of the definition of these factors given by Serre (\cf \cite{Se3}) as a product of powers of shifted $\Gamma$-functions, these functions are completely specified by the multiplicities of their poles at some integer points on the real line. By a result of Beilinson (\cite{Beilinson}, \cf also \cite{Deninger91} and \cite{Schneider}) the order of these poles can be expressed in terms of the ranks of some real Deligne cohomology groups of $X_\nu$, they only occur at integers $s=m\leq \frac w2$ and their multiplicity is provided by the well-known formula
\begin{equation}\label{ordpole}
    \displaystyle{{\rm ord}_{s=m}}L_\nu(H^w(X),s)^{-1}={\rm dim}_\R H^{w+1}_{\cD}(X_{\nu},\R(w+1-m))
\end{equation}
where for a complex place $\nu$ the real Deligne cohomology groups of $X_\nu$ are defined as in \S \ref{redD} ($X_{\nu}=X_{\C}$) and for a real place $\nu$ ($X_\nu = X_{/\R}$), these groups are defined as
\begin{equation}\label{realrealdefn}
   H^q_{\cD}(X_{\nu},\R(p)):=H^q_{\cD}(X_\nu\otimes_\R\C,\R(p))^{\rm DR-conjugation}
\end{equation}
\ie the subspace of $H^q_{\cD}(X_\nu\otimes_\R\C,\R(p))$ made by fixed points of  the de Rham conjugation.

\subsection{Deligne cohomology and archimedean cohomology}

 Let us now consider the pairs $(w,m)$ of integers which enter in formula \eqref{ordpole}, when $\dim~ X=d$  (so that $w$ takes integer values between $0$ and $2d$). They give rise to the set (here we use $w=q$)
  \begin{equation}\label{ad}
 A_d=\{(q,m)\mid 0\leq q\leq 2d, \ m\leq q/2\}.
  \end{equation}
  The relevant infinite dimensional archimedean cohomology of $X_{\nu}$  is therefore dictated by \eqref{ordpole} as the infinite direct sum
\begin{equation}\label{relevcohom}
\bigoplus_{(q,m)\in A_d} H^{q+1}_{\cD}(X_{\nu},\R(q+1-m)).
\end{equation}
\begin{figure}
\begin{center}
\includegraphics[scale=0.6]{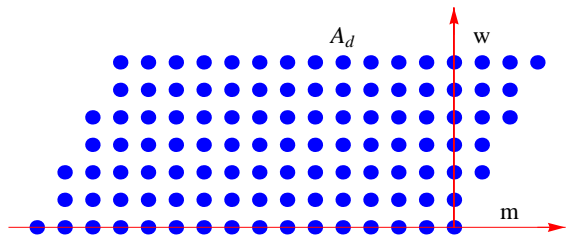}
\caption{The subset $A_d$ for $d=3$}\label{adsub}
\end{center}
\end{figure}
\begin{figure}
\begin{center}
\includegraphics[scale=0.6]{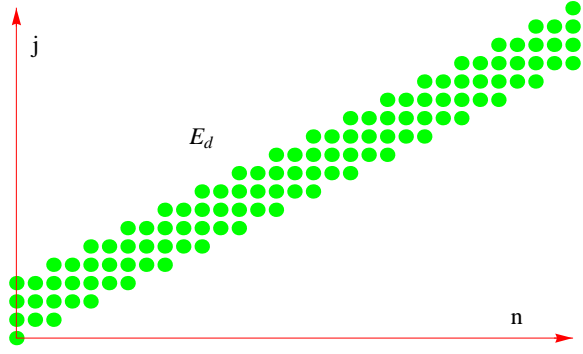}
\caption{The subset $E_d$ for $d=3$}\label{edsub}
\end{center}
\end{figure}
We shall now rewrite this sum in an equivalent way by using the following
\begin{lem}\label{trivchange}
Let $d\geq 0$ be an integer. The map of sets which maps the pair of integers $(n,j)$ to $(q,m)$ under the relations
\begin{equation*}\label{njqm}
    q=2j-n,\ \ m=j-n
\end{equation*}
is a bijection of $E_d=\{(n,j)\mid n\geq 0,\, 0\leq 2j-n\leq 2d\}$ with
$A_d$.
\end{lem}
\proof The inverse map sends $(q,m)\in A_d$ to $(n,j)$ where $n= -2 m+q,j= -m+q$. One easily checks that the conditions are preserved, in fact one has
\begin{equation*}
  n\geq 0, \, 0\leq 2j-n\leq 2d\implies
  \ j-n\leq (2j-n)/2
\end{equation*}
and
\begin{equation*}
   0\leq q\leq 2d, \ m\leq q/2\implies
  -2 m+q\geq 0, \, 0\leq 2(-m+q)-(-2 m+q)\leq 2d.
\end{equation*}
 \endproof

Thus by Lemma \ref{trivchange}, the sum in \eqref{relevcohom} can be equivalently rewritten as
\begin{equation}\label{change}
  \bigoplus_{(q,m)\in A_d} H^{q+1}_{\cD}(X_{\nu},\R(q+1-m))=\bigoplus_{(n,j)\in E_d}\, H^{2j+1-n}_{\mathcal D}(X_{\nu},\R(j+1)).
\end{equation}
Moreover, by applying the Proposition \ref{propcyclic1}, one obtains the isomorphism
\begin{equation}\label{change1}
\bigoplus_{n\geq 0} HC_n(X_\C)\cong\bigoplus_{(n,j)\in E_d}\,\tilde H^{2j+1-n}_{\mathcal D}(X_\C,\R(j+1)).
\end{equation}
It follows that the difference between the archimedean cohomology of $X_{\nu}$ as in \eqref{change} and the cyclic homology direct sum as in \eqref{change1} is expressed by the nuance between the real Deligne cohomology  and the reduced  Deligne cohomology of $X_\C$.

\subsection{Passing from $\tilde H^{*}_{\mathcal D}(X_\nu,\R(\cdot))$ to $H^{*}_{\mathcal D}(X_{\nu},\R(\cdot))$}\label{passingsect}

To understand the difference between reduced Deligne cohomology $\tilde H^{*}_{\mathcal D}(X_\nu,\R(\cdot))$  and real Deligne cohomology $H^{*}_{\mathcal D}(X_{\nu},\R(\cdot))$, we assume first that the place $\nu$ is complex, and introduce the long exact sequence associated to the short exact sequence \eqref{deligne0}. It is of the form
\begin{equation*}
\cdots\to H^w_B(X(\C),\R(r))\to \tilde H^{w+1}_{\mathcal D}(X_\C,\R(r))\to H^{w+1}_{\mathcal D}(X_{\C},\R(r)) \to H^{w+1}_B(X(\C),\R(r))\to\cdots
\end{equation*}
This long exact sequence can be equivalently written as
\begin{equation*}
\cdots\to H^w_B(X(\C),\R(r))\to  H_{{\rm dR},\text{rel}}^{w}(X_\C,r)\to H^{w+1}_{\mathcal D}(X_\C,\R(r)) \to H^{w+1}_B(X(\C),\R(r))\to\cdots
\end{equation*}
 in view of the isomorphism $\tilde H^{w+1}_{\mathcal D}(X_\C,\R(r))\cong H_{{\rm dR},\text{rel}}^w(X_\C,r):=H^w_{{\rm dR}}(X(\C))/F^r$. By comparing with \eqref{change}, we see that one has to estimate  $H^{w+1}_{\mathcal D}(X_\C,\R(r))$  for $(w,m)\in A_d$, with $A_d$ as in \eqref{ad} where $r=w+1-m$, \ie for $(w,w+1-r)\in A_d$.
One knows that for $w<2r$ the natural map
\begin{equation*}
    H^w_B(X(\C),\R(r))\to  H^w_{{\rm dR}}(X(\C))/F^r
\end{equation*}
is injective (in $H^w_{{\rm dR}}$ the intersection $F^r\cap \bar F^r=\{0\}$ when $2r>w$). For $w+1<2r$ one gets a short exact sequence of the form
\begin{equation}\label{seqshort}
 0\to H^w_B(X(\C),\R(r))\to  \tilde H^{w+1}_{\mathcal D}(X_\C,\R(r))\to H^{w+1}_{\mathcal D}(X_\C,\R(r)) \to 0
\end{equation}
This holds for the pair $(w,w+1-r)\in A_d$ since
\begin{equation*}
    (w,w+1-r)\in A_d\implies w+1-r\leq w/2 \implies w/2+1\leq r\implies w+1<2r.
\end{equation*}
Therefore,when the place $\nu$ is complex, the difference between the archimedean cohomology of $X_{\nu}$ as in \eqref{change} and the cyclic homology direct sum as in \eqref{change1} will be taken care of by a suitable interpretation in cyclic terms of the real Betti cohomology $H^{w}_B(X(\C),\R(r))$ and of the map  $H^w_B(X(\C),\R(r))\to  \tilde H^{w+1}_{\mathcal D}(X_\C,\R(r))$.
When instead the place $\nu$ of $K$ is real  the corresponding real Deligne cohomology groups are defined in \eqref{realrealdefn} as fixed points under the anti-linear de Rham conjugation $\bar F_\infty$. In \S \ref{sectreal} we shall show that  $\bar F_\infty$ admits a direct cyclic homology interpretation. Thus at a real place $\nu$ one has to further refine the above construction by taking the fixed points of the cyclic counterpart of $\bar F_\infty$. This is how the remaining part of the paper develops.

\section{Real ($\lambda$)-twisted cyclic homology of a smooth, projective variety}\label{sectrealcyc}

In this section we give the interpretation in cyclic terms of the real Betti cohomology and of the map  $H^w_B(X(\C),\R(r))\to  \tilde H^{w+1}_{\mathcal D}(X_\C,\R(r))$. We show that the computation of the cyclic homology of smooth manifolds (\cf \cite{ncg}, Theorem 46) jointly with the stability property of the periodic cyclic homology when passing from the algebraic cyclic homology of a smooth complex projective variety $X_\C$ to the associated $C^\infty$-manifold $X_\sss$ (using a natural morphism of schemes $\Spec(C^\infty(X_\sss,\C))\to X_\C$ whose definition is given in Appendix~\ref{sectweil}), give rise (with the implementation of a natural involution on the Frechet algebra $C^\infty(X_\sss,\C)$  and of a Tate twist) to a {\em real subspace} $HP^\rrr_*(X_\C)$ of the periodic cyclic homology of $X_\C$. This new structure is more precisely described by means of a map of {\em real graded vector spaces}
\begin{equation*}\label{realst}
HP^\rrr_*(X_\C)\stackrel{\tau}{\to} HP_*(X_\C)
\end{equation*}
which we deduce from a map of associated complexes.

\subsection{Hochschild and cyclic homology of schemes}

We first recall from \cite{weibel} the definition of Hochschild and cyclic homology of a scheme $X_k$ over a field $k$. We use the general conventions of \cite{weibel} and re-index a chain complex $\cC_*$ as a cochain complex by writing  $\cC^{-n}:=\cC_n$, $\forall n\in \Z$.

\subsubsection{Hochschild homology}

One lets $\cC_*^h$ be the sheafification of the chain complex of presheaves $U\mapsto C_*^h(\cO_{X_k}(U))=\cO_{X_k}(U)^{\otimes (*+1)}$, (the tensor products are over $k$), with the Hochschild boundary $b$ as differential. Then, one re-indexes this complex as a negative (unbounded) cochain complex \ie one lets $\cC^{-n}:=\cC_n^h$, $\forall n$, and one finally defines the Hochschild homology $HH_n(X_k)$ of $X_k$  as the Cartan-Eilenberg hypercohomology $\mathbb H^{-n}(X_k,\cC^{*})$
\[
HH_n(X_k) := \mathbb H^{-n}(X_k,\cC^{*}) = H^{-n}(\Gamma(\tot~ I^{**})).
\]
Here $I^{**}$ is an injective Cartan-Eilenberg resolution of  $\cC^{*}$ and $\tot$ is the total complex (using products rather than sums, \cf \cite{weibelbook}, 1.2.6 and \cite{weibel}, Appendix).
\begin{figure}
\begin{center}
\includegraphics[scale=0.6]{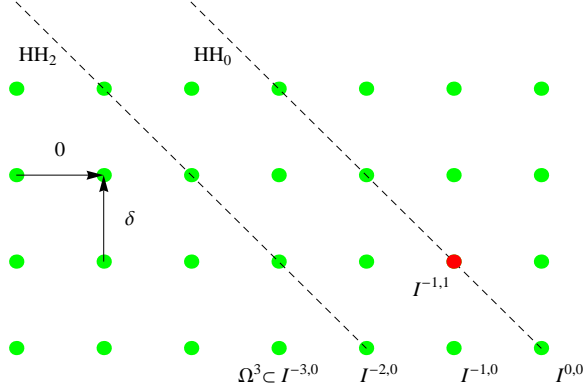}
\caption{Bi-complex describing the Hochschild homology of  $X_\C$, the vertical line $(I^{-m,*},\delta)$ with abscissa $-m$ is an injective resolution of $\Omega^m$. }\label{hochsheaf}
\end{center}
\end{figure}
We shall only work with algebras over $k=\C$ or $\R$ and in that case the $\lambda$-decomposition (for commutative algebras) provides a natural decomposition of $\cC_*$ as a direct sum of chain sub-complexes of the form
\begin{equation*}\label{ccc}
    \cC_*=\oplus\, \cC_*^{(i)}
\end{equation*}
and one then sets $HH_n^{(i)}(X_k) := \mathbb H^{-n}(X_k,\cC_{*}^{(i)})$.\newline
In the case of affine schemes $X_k=\Spec(A)$ for a $k$-algebra $A$,  one has the equality with Hochschild homology (\cf \cite{WG}, 4.1, and \cite{Weibelcris} Proposition 1.3)
\begin{equation*}\label{hochs}
    HH_n(X_k)=HH_n(A), \ \ HH_n^{(i)}(X_k)=HH_n^{(i)}(A).
\end{equation*}
Let now $X_\C$ be the scheme  over $\C$ associated to a smooth, projective complex (classical) algebraic variety $X_\aaa$. In this case even though the scheme $X_\C$ has more points than the variety $X_\aaa$ the categories of open sets are the same and so one can talk indifferently about sheaves over $X_\aaa$ or over $X_\C$. By applying \cite{Weibelcris} (Corollary 1.4) one has
\begin{equation*}\label{hochs1}
   HH_n^{(j)}(X_\C)\cong H^{j-n}(X_\C,\Omega_{X_\C}^j)\qqq j, n.
\end{equation*}
Figure \ref{hochsheaf} shows the corresponding simplified version of the Cartan-Eilenberg injective resolution of the Hochschild complex. Using the Hochschild-Kostant-Rosenberg theorem one simplifies the chain complex $\cC_*^h$  by replacing the sheaf $\cC_j^h$ with the sheaf $\Omega^j$ of algebraic differential forms of degree $j$ and the Hochschild boundary $b$ by $0$. One can then choose a
Cartan-Eilenberg injective resolution where the horizontal boundaries are $0$. The $\lambda$-decomposition is now read as the decomposition of the bi-complex in the sum of its vertical columns.

\subsubsection{Cyclic homology}\label{sectcychom}

 In order to recall the definition of the cyclic homology $HC_*(X_k)$ of a  scheme $X_k$ over a field $k$ given in \cite{Weibelcris}, we start by reviewing the following definition of cyclic homology of a mixed complex of sheaves $(\cM,b,B)$ on $X_k$
 \begin{equation*}\label{cycfirst}
    HC_n(\cM):=\mathbb H^{-n}(X_k,\tot \cB_{**}(\cM))
 \end{equation*}
 where $\cB_{**}(\cM)$ denotes the corresponding bi-complex of sheaves.
 One then applies this definition to the usual $(b,B)$-bicomplex, \ie to the mixed complex on $X_k$ where the $B$ co-boundary operator turns the sheafified Hochschild chain complex into a mixed complex of sheaves.

 In the case of an affine scheme $X_k=\Spec(A)$, for a $k$-algebra $A$ ($k=\C$ in our case),  one derives by applying the fundamental result of  \cite{weibel} (\cf Theorem 2.5) the equality with the usual cyclic cohomology
  \begin{equation*}\label{cycweib}
    HC_n(X_k)=HC_n(A), \ \ HC_n^{(i)}(X_k)=HC_n^{(i)}(A).
\end{equation*}
  Let now $X_\C$ be the scheme  over $\C$ associated to a smooth, projective complex (classical) algebraic variety $X_\aaa$. Next we describe the simplified version of the complex of sheaves
  $\tot \cB_{**}(\cM)$ using the afore-mentioned Hochschild-Kostant-Rosenberg theorem. This amounts to replace  the mixed complex $(\cM,b,B)$ of sheaves by the mixed complex $(\Omega_{X_\C}^*,0,d)$ where $d$ is the de Rham boundary (which corresponds to the coboundary operator $B$). The total complex of the $(b,B)$-bicomplex simplifies to the following chain complex of sheaves
  \begin{equation*}\label{chain}
   \cT_{m}=\bigoplus_{u\geq 0} \Omega_{X_\C}^{m-2u}, \ \
   (b+B)\left(\sum_{u\geq 0}\, \omega_{m-2u}\right)=\sum_{u\geq 1} \,d\omega_{m-2u}\in \cT_{m-1}.
  \end{equation*}
  As explained in \cite{Weibelcris} (Example 2.7),
  passing to the corresponding  cochain complex of sheaves indexed in negative degrees
    $\cT^{-n}:=\cT_{n}$ one obtains (using $j=-n-u$)
    \begin{equation*}
        \cT^n=\bigoplus_{ 0\leq j\leq - n} \Omega_{X_\C}^{2j+n}, \ \ (\cT^*,d)=
        \bigoplus_{ j\geq 0}(\Omega_{X_\C}^{2j+*},d)_{(j+*)\leq 0}=
        \bigoplus_{ j\geq 0}(\Omega_{X_\C}^{\leq j},d)[-2j]
    \end{equation*}
    which is the product of the truncated de Rham complexes $(\Omega_{X_\C}^{\leq j},d)$ shifted by $-2j$. Moreover this decomposition corresponds to the
    $\lambda$-decomposition of the mixed complex $(\cM,b,B)$.
    Using a Cartan-Eilenberg injective resolution $(A^{*,*};d,\delta)$ of the de Rham complex, one obtains a Cartan-Eilenberg injective resolution $(J^{*,*};d_1,d_2)$ of the cochain complex $\cT^*$ of the following form ($r,s\in\Z$)
\begin{equation}\label{jrs}
    J^{r,s}:=\bigoplus_{p\leq -r, \atop p\equiv r (2)} A^{p,s}=
    \bigoplus_{p\equiv r (2)} A^{p,s}/F^{-r+1}.
\end{equation}
This is the quotient of the strictly periodic bi-complex
$
(P^{r,s}=\bigoplus_{p\equiv r (2)} A^{p,s};d,\delta)
$
   \begin{figure}
\begin{center}
\includegraphics[scale=0.6]{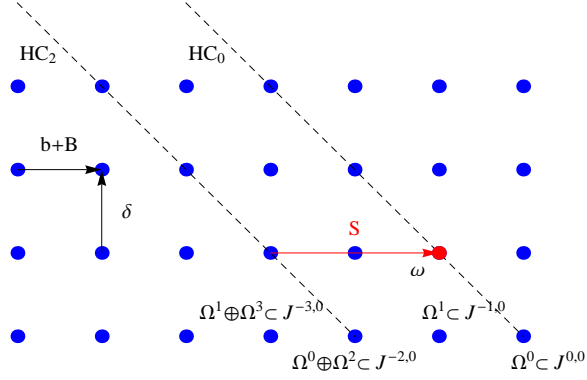}
\caption{Bicomplex describing the cyclic homology of  $X_\C$, and map $S$}\label{cycsheaf}
\end{center}
\end{figure}
by the sub-complex
\begin{equation}\label{jneg}
(N^{r,s}=\bigoplus_{p> -r, \atop p\equiv r (2)} A^{p,s};d,\delta)
 \end{equation}
 corresponding to the Hodge filtration.
The differentials in \eqref{jrs} are given by the truncation $d_1$ of the horizontal differential in the resolution $(A^{*,*};d,\delta)$ of the de Rham complex and by $d_2= \delta$. The $\lambda$-decomposition corresponds to the following decomposition of $J^{*,*}$
\begin{equation*}\label{lamdecj}
 J^{r,s}=\bigoplus_{j\geq 0} J^{r,s}(j), \ \    J^{r,s}(j)= \left\{
                                                              \begin{array}{ll}
                                                                A^{r+2j,s}, & \hbox{if} \  r+j\leq 0 ;\\
                                                                \{0\}, & \hbox{otherwise.}
                                                              \end{array}
                                                            \right.
\end{equation*}
Thus, since $r+j\leq 0\Longleftrightarrow r+2j\leq j$, $(J^{*,*}(j);d_1,d_2)$ provides a Cartan-Eilenberg resolution of the truncated de Rham complex $\Omega^{\leq j}_{X(\C)}$. By
    computing the hypercohomology, one obtains (\cf\cite{Weibelcris}, Theorem 3.3)
 \begin{equation*}\label{weibbis}
  HC_n^{(j)}(X_\C)= \mathbb H^{2j-n}(X_\C,\Omega^{\leq j}_{X_\C}) , \ \ HC_n(X_\C)=\prod_{j\in\Z} \mathbb H^{2j-n}(X_\C,\Omega^{\leq j}_{X_\C}) .
 \end{equation*}
 The periodicity operator $S$ acquires here a simple meaning as an endomorphism of the bi-complex $J^{*,*}$, as  Figure \ref{cycsheaf} shows. We shall describe this fact more in details in the next section.

\subsection{Resolution using differential forms of type $(p,q)$}\label{sectrespq}

Let $X_\C$ be a smooth, projective variety over $\C$. In Appendix \ref{sectweil} (\cf \S \ref{sectmorphis}), we define a canonical morphism of schemes $\pi_X:\Spec(C^\infty(X_\sss,\C))\to X_\C$. In order to understand the effect of $\pi_X$ on cyclic homology we use, as in \cite{Weibelcris}, the results of \cite{GAGA}
and \cite{Grothendieck} to pass from the Zariski to the analytic topology and then use the Dolbeault resolution as a soft resolution of the de Rham complex. Thus the general discussion as given in \S \ref{sectcychom} applies using for  $A^{p,q}$  the $\C$-vector space of smooth complex differential forms on $X(\C)$ of type $(p,q)$. Thus the cohomology $\mathbb H^{-n}$ of the following bi-complexes ($r,s\in\Z$)
computes the cyclic homologies $HC_n(X_\C)$ \resp $HC_n^{(j)}(X_\C)$
\begin{equation}\label{crs}
    C^{r,s}:=\bigoplus_{p\leq -r, \atop p\equiv r (2)} A^{p,s}, \qquad
C^{r,s}(j):= \left\{
                                                              \begin{array}{ll}
                                                                A^{r+2j,s}, & \hbox{if} \  r+j\leq 0 ;\\
                                                                \{0\}, & \hbox{otherwise.}
                                                              \end{array}
                                                            \right.
\end{equation}
The differentials are defined as follows: $d_1=\partial$ and $d_2=\bar \partial$, where $d_1=\partial$ is truncated
according to the identification of $C^{*,*}$ with the quotient of the strictly periodic bicomplex
by the sub-complex $N^{r,s}$ corresponding to the Hodge filtration as in \eqref{jneg}. The periodic bi-complex and its $\lambda$-decomposition are given by
\begin{equation}\label{perbis}
(P^{r,s}=\bigoplus_{p\equiv r (2)} A^{p,s};\partial,\bar\partial), \qquad  P^{r,s}(j)= A^{r+2j,s}\qqq r,s,j\in \Z.
 \end{equation}
 Next result shows an equivalent analytic way to describe the cyclic homology of $X_\C$. As usual, $T^*_\C$ denotes the complexified cotangent bundle on the manifold $X(\C)$
\begin{lem}\label{lemper}
Let $X_\C$ be a smooth complex projective  variety of dimension $m$. Then in degree $\leq -m$ the total complex ${\rm Tot}\,C^{*,*}$ of \eqref{crs} coincides with the complex of smooth differential forms of given parity on the associated complex manifold $X(\C)$
\begin{equation}\label{par}
   \cdots\stackrel{d}{\to} C^\infty(X(\C),\wedge^{\rm even}T^*_\C)\stackrel{d}{\to}
    C^\infty(X(\C),\wedge^{\rm odd}T^*_\C)\stackrel{d}{\to}\cdots
\end{equation}
where $d$ is the usual differential. The total complex ${\rm Tot}\,P^{*,*}$ of \eqref{perbis} coincides with \eqref{par}, and the total complex ${\rm Tot}\,P^{*,*}(j)=PC^*(j)$ is given by the sub-complex
\begin{equation}\label{parj}
   \cdots\stackrel{d}{\to} C^\infty(X(\C),\wedge^{*+2j}T^*_\C)\stackrel{d}{\to}
    C^\infty(X(\C),\wedge^{*+1+2j}T^*_\C)\stackrel{d}{\to}\cdots
\end{equation}
\end{lem}
\begin{figure}
\begin{center}
\includegraphics[scale=0.6]{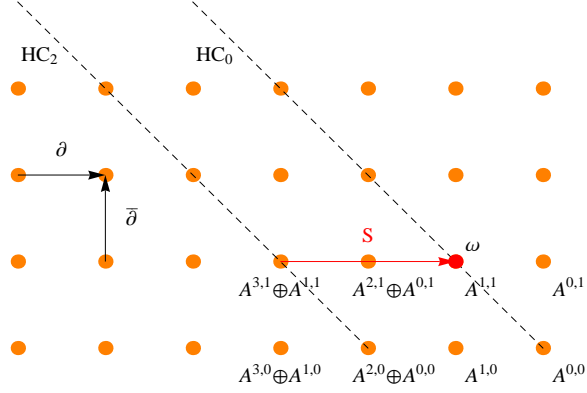}
\caption{The bi-complex $(C^{*,*},\partial,\bar\partial)$ computing the cyclic homology of $X_\C$}\label{cycforms}
\end{center}
\end{figure}
\proof The component of degree $k$ of the total complex ${\rm Tot}\,C^{*,*}$ is described by
\[
\bigoplus_{r+s=k} C^{r,s}=\bigoplus_{r+s=k} \bigoplus_{p\leq -r\atop p\equiv r (2)} A^{p,s}.
\]
If $k\leq -m$ then $r\leq -m$ and the condition $p\leq -r$ is automatic so that one derives
\[
\bigoplus_{r+s=k} C^{r,s}=\bigoplus_{r+s=k} \bigoplus_{ p\equiv r (2)} A^{p,s}
=\bigoplus_{p+s\equiv k (2)}A^{p,s}=C^\infty(X(\C),\wedge^{k ({\rm mod}\,2)} T^*_\C).
\]
Moreover the co-boundary operator of ${\rm Tot}\,C^{*,*}$ is $\partial+\bar\partial=d$ so the statement follows. The remaining statements follow from \eqref{perbis}.\endproof

 Next, we describe in more details the action of the periodicity map $S$ of cyclic homology on the bi-complex $(C^{*,*};\partial,\bar\partial)$ as in \eqref{crs}.
\begin{lem}\label{mapS}
The following map defines an endomorphism of degree $2$ of the bi-complex $(C^{*,*};\partial,\bar\partial)$
\begin{equation}\label{maps}
    S\omega =\omega/F^{-r-1}\in C^{r+2,s}\qqq  \omega \in C^{r,s}.
\end{equation}
\end{lem}
\proof It is enough to show that the  sub-complex $N^{r,s}$ as in \eqref{jneg} of the periodic bi-complex \eqref{perbis} is stable under the map $S$. A differential form $\omega$ belongs to $N^{r,s}\subset P^{r,s}$ if and only if its homogeneous components do so. Thus, for $\omega$ of type $(p,s)$ this means that $p\equiv r$ (mod $2$) and that $p\geq -r+1$. The element $S\omega$  is of type $(p,s)$ and one needs to check that it belongs to $N^{r',s}$ with $r'=r+2$. One has
$p\equiv r'$ (mod. $2$) and  $p\geq -r'+1$. \endproof

Correspondingly, the bi-complex  that computes the Hochschild homology of $X_\C$ is described in Figure \ref{hochsforms}. It fills up a finite square in the plane and it is defined by
\[
C_h^{r,s}:=(A^{-r,s};0,\bar\partial).
 \]
This bi-complex maps by inclusion into the bi-complex $(C^{*,*},\partial,\bar\partial)$ and it coincides precisely with the kernel (bi-complex) of the map $S$. Thus, by using the surjectivity of $S$ at the level of complexes, one derives an exact sequence of associated total complexes
\begin{equation}\label{exactseq}
    0\to C_h^*\to C^*\stackrel{S}{\to} C^{*+2}\to 0.
\end{equation}
The corresponding long exact sequence of cohomology groups is then the $SBI$ exact sequence
\begin{equation*}\label{sbi}
    \cdots\to HH_n(X_\C) \stackrel{I}{\to} HC_n(X_\C)\stackrel{S}{\to}HC_{n-2}(X_\C)\stackrel{B}{\to} HH_{n-1}(X_\C)\to\cdots
\end{equation*}
\begin{figure}
\begin{center}
\includegraphics[scale=0.6]{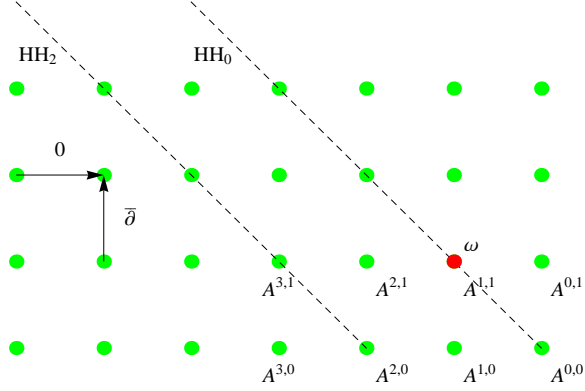}
\caption{The bi-complex $C_h^{r,s}=(A^{-r,s},0,\bar\partial)$ computing the Hochschild homology of $X_\C$}\label{hochsforms}
\end{center}
\end{figure}
\begin{rem} \label{ssurject}{\rm
Since the analytic space  $X_\ann$  associated to $X_\C=X_\aaa$ (\cf \cite{GAGA}) can be endowed with the structure of  a  K\"{a}hler manifold, the map $S$ is surjective at the level of cyclic homology (\cf \cite{Weibelcris} Proposition 4.1). We discuss (in the K\"{a}hler case) a specific example
where the result might look surprising. We consider a cyclic homology class $[\omega]\in HC_0(X_\C)$ coming from a form $\omega$ of type $(1,1)$ (\cf the red dot in Figure \ref{cycforms}). One has $\bar\partial \omega=0$ although there is no condition on $\partial \omega$ since the horizontal co-boundary, is automatically zero  due to the truncation.  It seems surprising that $[\omega]$ can belong to the image of $S$ since this means exactly that there exists a form $\omega'$ also of type $(1,1)$ and also representing the above class (\ie $[\omega]=[\omega']\in HC_0(X_\C)$) and such that
\begin{equation*}
   \bar\partial \omega'=0=\partial \omega'.
\end{equation*}
The reason why this holds is that one can modify $\omega$ to $\omega'=\omega+\bar\partial \alpha$ with $\alpha$ of type $(1,0)$ by using the K\"{a}hler metric so that $\omega'$ becomes harmonic  for any choice of the three Laplacians
\begin{equation*}
    \bar\partial \bar\partial^*+\bar\partial^*\bar\partial=
    \partial \partial^*+\partial^*\partial=\frac 12( dd^*+d^*d).
\end{equation*}
Then it follows that $\omega'$ is automatically also $\partial$-closed.
}\end{rem}

\subsection{The periodic cyclic homologies $HP_*(X_\C)$ and $HP_*(C^\infty(X_\sss,\C))$}\label{pch}

We shall now compare  the periodic cyclic homology $HP_*(X_\C)$ of a smooth and complex projective  variety $X_\C$ (viewed as a scheme over $\C$) with the periodic cyclic homology $HP_*(C^\infty(X_\sss,\C))$ of the underlying  smooth manifold $X_\sss=X(\C)$.\newline
By \cite{ncg} (\cf Theorem 46), the periodic cyclic homology of the commutative algebra $C^\infty(X_\sss,\C)$ is computed as the de Rham cohomology of smooth differential forms, \ie as the Betti cohomology with complex coefficients $H^*_B(X(\C),\C)$. This requires to take care of the natural Frechet topology on the algebra $C^\infty(X_\sss,\C)$. By Lemma 45 of \cite{ncg} one derives the description of the topological Hochschild homology groups as $HH_k(C^\infty(X_\sss,\C))=C^\infty(X_\sss,\wedge^kT^*_\C)$ where the map $B$ is the de Rham differential. We refer to \cite{ncg} for the computation of the cyclic homology and in particular for the identification of the periodic cyclic homology with the de Rham cohomology of smooth differential forms. Moreover, by applying Theorem 4.6.10 of \cite{Loday} (adapted to the topological case) one obtains the more precise identification
\begin{equation*}\label{cyctop}
    HP_n^{(j)}(C^\infty(X_\sss,\C))\cong H^{2j-n}_{B}(X(\C),\C)\qqq n,j\in \Z
\end{equation*}
where one takes into account the natural Frechet topology on the algebra $C^\infty(X_\sss,\C)$ in the periodic cyclic homology.
\begin{prop}\label{lemperbis}
Let $X_\C$ be a smooth complex projective  variety of dimension $m$.\newline
$(i)$~The map $\pi_X^*$ of cyclic homology groups induced by the morphism of schemes \[
\pi_X:\Spec(C^\infty(X_\sss,\C))\to X_\C
\]
is the composite of the following two isomorphisms
\begin{equation*}
 \,\hspace{100pt}\raisetag{-50pt} \xymatrix@C=25pt@R=25pt{
 HP_n^{(j)}(X_\C)\ar[dr]_\sim^{\varsigma}\ar[rr]^-{\pi_X^*} & &
 HP_n^{(j)}(C^\infty(X_\sss,\C)) & \\
 &  H^{2j-n}_{B}(X(\C),\C) \ar[ur]^\sim &\\
}\hspace{140pt}
\end{equation*}
$(ii)$~For a pair of integers $n,j$ with $n\geq 2m$ and $\frac n2\leq j\leq n$, the statement  in $(i)$ holds at the level of the cyclic homology groups, \ie
\begin{equation*}\label{percycl}
    HC_n^{(j)}(X_\C)\cong H^{2j-n}_{B}(X(\C),\C)\cong HC_n^{(j)}(C^\infty(X_\sss,\C))
\end{equation*}
where the upper index $(j)$ refers to the corresponding component of the $\lambda$-decomposition.
\end{prop}
\proof $(i)$~Take a covering $\cU$ of $X_\C$ made by affine Zariski open sets $U_i$. The morphism $\pi_X$ is affine (as any morphism of an affine scheme to a projective scheme) thus the inverse images $V_i=\pi_X^{-1}(U_i)$ form an open affine covering $\cV$ of the scheme $\Spec(C^\infty(X_\sss,\C))$. The discussion of  \S \ref{sectcychom} applies to $X_\C$ using in place of the injective resolution $(A^{*,*};d,\delta)$ of the de Rham complex, the \v{C}ech bicomplex (\cf \cite{Hart} Lemma III 4.2 and Theorem III 4.5)
\begin{equation*}\label{cech}
    (\text{\v{C}}^{p,q}=C^q(\cU,\Omega^p),d,\delta)
\end{equation*}
where $d$ is the de Rham coboundary and $\delta$ the \v{C}ech coboundary. In particular $HP_n^{(j)}(X_\C)$ is the cohomology $H^{-n}$ of the complex of cochains
\begin{equation*}\label{drequ}
   (\bigoplus_{r+s=*} \text{\v{C}}^{r+2j,s}, d+\delta).
\end{equation*}
 Let then $\omega\in HP_n^{(j)}(X_\C)$ be represented by a cocycle
 \begin{equation*}
 \omega=\sum_{r+s=-n}\omega_{r,s}\,, \ \
 \omega_{r,s}\in C^s(\cU,\Omega^{r+2j})
 \end{equation*}
 The image $\pi_X^*(\omega)\in HP_n^{(j)}(C^\infty(X_\sss,\C))$ is then given by the corresponding \v{C}ech cocycle for the affine open covering $\cV$ of $\Spec(C^\infty(X_\sss,\C))$.
Now, by Theorem \ref{affscheme} $(iii)$, the map $\pi_X^*$ coincides on the affine Zariski open sets $U:=U_{i_0,\ldots ,i_k}=\cap U_{i_j}$ with the inclusion $\Gamma(U,\Omega_{X_\C}^p)\stackrel{\iota}{\to} C^\infty(U,\wedge^{p,0}T^*_\C)$ of algebraic sections of  $\Omega_{X_\C}^p$ into the space of smooth sections  of the vector bundle $\wedge^{p,0}T^*_\C$ of (complex) differential forms of type $(p,0)$. This inclusion is the restriction of the corresponding inclusion for the analytic space $X_\ann$ and the latter is a morphism for the following two resolutions of the constant sheaf $\C$ for the usual topology. The first resolution is given by the sheaf of holomorphic differential forms
\begin{equation*}
    0\to \C\to \Omega^0 \stackrel{\partial}{\to} \Omega^1\to \ldots
\end{equation*}
and the second resolution is the de Rham complex of sheaves of $C^\infty$ differential forms
\begin{equation*}
    0\to\C \to C^\infty(\bullet,\wedge^0 T^*_\C)\stackrel{d}{\to}
    C^\infty(\bullet,\wedge^1 T^*_\C)\to \ldots
\end{equation*}
This shows that  the Betti cohomology class of $\omega$ is the same as the Betti cohomology class of $\pi_X^*(\omega)$ which is represented by the \v{C}ech cocycle $\iota(\omega)$ in the \v{C}ech bicomplex $C^{a,b}=C^{b}(\cV,\wedge^a T^*_\C)$ of smooth differential forms, associated to the covering $\cV$. In this bicomplex the vertical lines are contractible since one is in the affine case and thus $\iota(\omega)$  is cohomologous to a global section \ie a closed differential form of degree $2j-n$.

$(ii)$~For $n\geq 2m$, one sees that  $HC_n(C^\infty(X_\sss,\C))=\oplus_{j\geq n/2} H^{2j-n}_{B}(X(\C),\C)$.
The  statement then follows from the $\lambda$-decomposition  $HC_n^{(j)}(X_\C)$ by applying Theorem 3.4 of \cite{Weibelcris}  and Theorem 4.6.10 of \cite{Loday} (adapted to the topological case) for $HC_n^{(j)}(C^\infty(X_\sss,\C)).$ \endproof

\begin{rem}\label{unstableran}{\rm
$(i)$~To understand the behavior of the map $\pi_X^*$ we give a simple example. Let $X_\C=\P^1_\C$ be the complex projective line. It is obtained  by gluing two  affine lines $U_+=\Spec(\C[z])$ and
$U_-=\Spec(\C[1/z])$ on the intersection $U=\Spec(\C[z,1/z])$. Let $\cU=\{U_\pm\}$ be the affine covering of $X_\C$, then the differential form $dz/z\in C^1(\cU,\Omega^1_{X_\C})$ determines a cocycle $\omega$ in the  \v{C}ech bicomplex. Let $[\omega]_{\rm cyc}\in HC_0(X_\C)$ be the corresponding cyclic homology class as in Figure \ref{cycsheaf}. One has $[\omega]_{\rm cyc}\in HC^{(1)}_0(X_\C)$ by construction. Since for any commutative algebra $B$ ($B=C^\infty(X_\sss,\C)$ in this example) one has $HC_0(B)=HC_0^{(0)}(B)$ (\cf \cite{Loday} Theorem 4.6.7), it follows that $\pi_X^*([\omega]_{\rm cyc})\in
HC_0^{(1)}(B)=\{0\}$. Let now $[\omega]_{\rm per}\in HC_2(X_\C)$ be represented by the same cocycle $\omega$ in the \v{C}ech bicomplex, with $S [\omega]_{\rm per}=[\omega]_{\rm cyc}$ as in Figure \ref{cycsheaf}. One has $[\omega]_{\rm per}\in HC_2^{(2)}(X_\C)$ and $\pi_X^*([\omega]_{\rm per})\in HC_2^{(2)}(B)$ is obtained by first writing the cocycle $\omega$ as a coboundary $\delta(\xi)$ in the \v{C}ech bicomplex of $B$,
\begin{equation*}
 \omega=\xi_+-\xi_-,\ \     \xi_+=\bar z dz/(1+z\bar z), \ \ \xi_-=- dz/(z(1+z\bar z)),
\end{equation*}
where $\xi_\pm\in \Gamma(\pi_X^{-1}(U_\pm),\Omega^1_B)$. Hence the class $\pi_X^*([\omega]_{\rm per})\in HC_2^{(2)}(B)$ is represented by the two form $\omega_2=-d\xi\in \Gamma(\Spec(B),\Omega^2_B)$
\begin{equation*}
     \omega_2=-dz\wedge d\bar z/(1+z\bar z)^2\in \Gamma(\Spec(B),\Omega^2_B).
\end{equation*}
 A similar computation performed in the \v{C}ech bicomplex of $B$, starting with $\pi_X^*([\omega]_{\rm cyc})$, produces $0$ because the coboundary $d$ is $0$ since there is no component $\Omega^2$ in $J^{0,0}$ (\cf Figure \ref{cycsheaf}).

$(ii)$~In the  above example $(i)$, it suffices to adjoin the variable $\bar z$ in order to obtain the existence of global affine coordinates such as  $z/(1+z\bar z)$, an affine scheme $Y$ and morphism $\pi:Y\to X_\C$ so that $\pi^*([\omega]_{\rm per})\in HC_2^{(2)}(\cO(Y))$ is represented by a global two form. This   illustrates the general fact that the map $\pi_X$ factorizes through an affine scheme ${}_*X$ whose construction is entirely in the realm of algebraic geometry. The general construction of the functor $X\mapsto {}_*X$ is provided in Appendix \ref{sectweil}.

$(iii)$~Lemma \ref{lemper} shows that the cyclic homology $HC_n(X_\C)$ already stabilizes, \ie becomes periodic,  for $n\geq m$. On the other hand, the cyclic homology $HC_n(C^\infty(X_\sss,\C))$ of the underlying smooth manifold only stabilizes for $n\geq 2m=2~\dim X_\C$, and is infinite dimensional for lower values of $n$. This shows that these two homologies cannot coincide in the unstable region $n<2m$.
}\end{rem}

\subsection{The Tate-twisted map $\tau: PC_\rrr^*(X_\C)\to PC^*(X_\C)$}\label{realperiodic}

We keep denoting with $X_\C$  a smooth, projective algebraic variety over $\C$ and with $X_\sss$ the associated smooth manifold. Let $PC^*(X_\C)$ and $PC^*(C^\infty(X_\sss,\C))$ be the cochain complexes which compute the periodic cyclic homologies as
\begin{equation}\label{negaindex}
    HP_n(X_\C)=H^{-n}\left(PC^*(X_\C)\right)   \,,  \ \ HP_n(C^\infty(X_\sss,\C))=H^{-n}\left(PC^*(C^\infty(X_\sss,\C))\right).
\end{equation}
In order to define the Tate twisted map
\begin{equation}\label{ttwist}
   \tau: HP^\rrr_*(X_\C)\to HP_*(X_\C)
\end{equation}
at the level of complexes of cochains, we shall use a quasi pullback of  two maps of complexes which correspond to
the morphisms of schemes
\begin{equation*}\label{schemaps}
    \pi_X:\Spec(C^\infty(X_\sss,\C))\to X_\C, \ \ \iota: \Spec(C^\infty(X_\sss,\C))\to \Spec(C^\infty(X_\sss,\R))
\end{equation*}
where the morphism $\iota$ is induced by the natural inclusion $C^\infty(X_\sss,\R)\subset C^\infty(X_\sss,\C)$.

We recall that if $A\stackrel{f}{\to} C\stackrel{g}{\leftarrow} B$ is a diagram of cochain complexes in an abelian category, its quasi-pullback is defined as the cochain complex
\begin{equation*}
A\times_C B:=\text{Cone}(A\oplus B\stackrel{f-g}{\to}C)[1].
\end{equation*}
The two natural projections give rise to maps from $A\times_C B$ to $A$ and $B$ \resp and the diagram
\[
\begin{CD}
A\times_C B@>f'>> B\\
@VVg'V @VVgV\\
A@>>f> C
\end{CD}
\]
commutes up-to canonical homotopy. The short exact sequence of complexes
\begin{equation}\label{basesequ}
 0\to \text{Cone}(A \stackrel{f}{\to}C)[1]\to
A\times_C B
 \stackrel{f'}{\to} B\to 0
\end{equation}
then determines an induced long exact sequence of cohomology groups. In particular when $f$ is a quasi isomorphism, so is $f'$ and the maps $g:B\to C$ and $g':A\times_C B\to A$ are quasi isomorphic.

\begin{defn}\label{maptau}
We define $PC_\rrr^*(X_\C)$ to be the  quasi pullback of the two morphisms of complexes
\begin{equation*}\label{qpb}
    \pi_X^*:PC^*(X_\C)\to PC^*(C^\infty(X_\sss,\C)),\quad (2\pi i)^{\Theta_0}:PC^*(C^\infty(X_\sss,\R))\to PC^*(C^\infty(X_\sss,\C))
\end{equation*}
where $\pi_X^*$ is the map of complexes induced from the morphism  $\pi_X$ and $\Theta_0$ is the generator of the $\lambda$-operations.
We denote by $\tau: PC_\rrr^*(X_\C)\to PC^*(X_\C)$ the projection map of complexes on the first factor of the quasi pullback.
\end{defn}
The map $\pi_X^*:PC^*(X_\C)\to PC^*(C^\infty(X_\sss,\C))$ is an isomorphism in cohomology (we take always into account the Frechet topology of $C^\infty(X_\sss,\C)$). Thus, the map $\tau: HP^\rrr_*(X_\C)\to HP_*(X_\C)$ is quasi isomorphic to the map $(2\pi i)^{\Theta_0}:PC^*(C^\infty(X_\sss,\R))\to PC^*(C^\infty(X_\sss,\C))$.
This shows, using the component $(j)$ in the $\lambda$-decomposition, that one  has a commutative diagram
\begin{equation}\label{Smap}
\begin{CD}
HP_{n+2}^{\rm real,\,(j+1)}(X_\C) @>S>> HP_n^{\rm real,\,(j)}(X_\C)\\
@VV\tau V @VV\tau V\\
HP_{n+2}^{(j+1)}(X_\C)@>>(2\pi i)^{-1}\,S> HP_n^{(j)}(X_\C)
\end{CD}
\end{equation}
In the simplest case of $X_\C$ be a single point, one has $HP_n^{(j)}(X_\C)\neq \{0\}$ only for $n=2j$ and in that case the map $\tau:HP_n^{\rm real,\,(j)}(X_\C)\to HP_n^{(j)}(X_\C)$ is the multiplication by $(2\pi i)^j$ from $\R$ to $\C$.
\begin{prop}\label{realiden} The isomorphism $\varsigma$ as in Proposition \ref{lemperbis} $(i)$ lifts to an isomorphism \[HP_n^{\rm real,\,(j)}(X_\C)\stackrel[\sim]{\varsigma}{\to} H_{B}^{2j-n}(X(\C),\R(j))\] such that the following diagram commutes
\begin{equation}\label{Smap1}
\begin{CD}
HP_n^{\rm real,\,(j)}(X_\C) @>{\varsigma} >\sim>H_{B}^{2j-n}(X(\C),\R(j))\\
@VV\tau V @VV\subset V\\
HP_n^{(j)}(X_\C) @>\varsigma>\sim > H_{B}^{2j-n}(X(\C),\C).
\end{CD}
\end{equation}
\end{prop}
\proof Taking the cohomology $H^{-n}$ and the component $(j)$ of the $\lambda$-decomposition, the map
\begin{equation*}
    (2\pi i)^{\Theta_0}:PC^*(C^\infty(X_\sss,\R))\to PC^*(C^\infty(X_\sss,\C))
\end{equation*}
is isomorphic to the map
\begin{equation*}
    H_{B}^{2j-n}(X(\C),\R(j))\to  H_{B}^{2j-n}(X(\C),\C)
\end{equation*}
associated to the additive group homomorphism  $\R(j)\subset \C$. \endproof

\section{The  archimedean cyclic homology $\har_*(X_\C)$}\label{sectharcyc}

In this section we give the definition and describe the first properties of the archimedean cyclic homology. For the construction, we follow the parallel with the description of the real Deligne cohomology that can be interpreted as the hyper-cohomology of a complex of sheaves which is, up to a shift of degrees, the quasi pullback of the two natural maps of complexes
\begin{equation*}
    \R(r)\to \Omega^*, \ \ F^r\Omega^*\to \Omega^*.
\end{equation*}
In cyclic homology, the Hodge filtration gets replaced by the negative cyclic homology: \cf \S \ref{sectneg}.  \S \ref{sectarcdefn} describes the archimedean cyclic homology of a smooth, projective complex algebraic variety.
In \S \ref{sectreal}, we give the corresponding definition in the case the algebraic variety is defined over the reals, by using the action of the anti-linear Frobenius.

\subsection{Negative cyclic homology}\label{sectneg}
We recall (\cf \cite{HoodJones}) that the negative cyclic homology $HN_*$ (of a cyclic $k$-module over a field $k$) is defined by dualizing the cyclic cohomology $HC^*$ with respect to the coefficient polynomial ring $k[v]=HC^*(k)$. Here, the variable $v$ has degree two and corresponds to the periodicity map $S$
in cyclic cohomology. At the level of the defining bi-complexes, the negative cyclic homology bi-complex $NC_{*,*}$ coincides, up-to a shift of degrees, with the kernel of the natural map of complexes connecting the periodic $(b,B)$ bi-complex $PC_{*,*}$ and the $(b,B)$ bi-complex $CC_{*,*}$. We set $u=v^{-1}$ and let $(C_*,b,B)$ be a mixed complex of chains. Then one finds (\cf \cite{HoodJones})
\begin{figure}
\begin{center}
\includegraphics[scale=0.5]{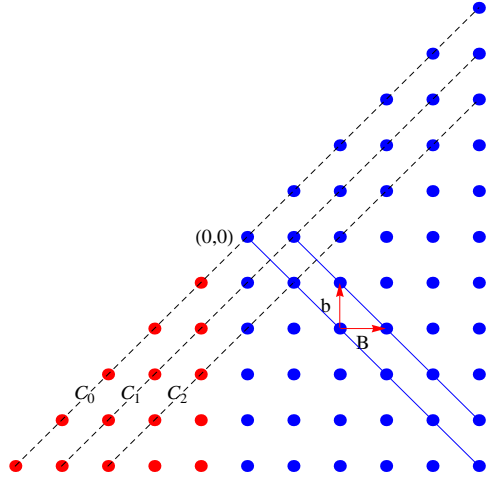}
\caption{In blue, the negative sub-complex $NC_{*,*}$ of the periodic $(b,B)$ bi-complex $PC_{*,*}$
all re-indexed in negative degrees. The hypercohomology $\mathbb H^{-n}$ of the blue complex of cochains is $HN_n$. The hypercohomology $\mathbb H^{-n}$ of the quotient (in red) complex of cochains is $HC_{n-2}$.}
\label{percyc}
\end{center}
\end{figure}
\begin{equation*}\label{bicomplexes}
    NC_{*,*}=k[u]\hat\otimes_k C_*, \qquad PC_{*,*}=k[u,u^{-1}]\hat\otimes_k C_*, \qquad
    CC_{*,*}=(k[u,u^{-1}]/uk[u])\hat\otimes_k C_*.
\end{equation*}
Here, the two boundary operators are $b$ and $uB$ and if $M$ and $N$ are graded $k$-modules, one defines $M\hat\otimes N$ by $(M\hat\otimes N)_n = \prod M_i\otimes N_{n-i}$.\newline
In the context of the sheafified theories over a scheme $X_k$ over a field $k$, it is convenient  (\cf  \cite{Weibelcris}) to re-index  the above bi-complexes in negative degrees, so that they turn into cohomology bi-complexes. Figure \ref{percyc} shows their behavior.\newline
Then the hyper-cohomology $\mathbb H^{-n}$ of the negative complex of cochains $NC^{*,*}$ is denoted by $HN_n(X_k)$, the hyper-cohomology $\mathbb H^{-n}$ of the periodic complex of cochains $PC^{*,*}$ is denoted $HP_n(X_k)$ and
the hyper-cohomology $\mathbb H^{-n}$ of the quotient cochain complex  is $HC_{n-2}(X_k)$. One  has a short  exact sequence of  cochain complexes  of the form
\begin{equation}\label{shortcomp}
    0\to NC^{*,*}\stackrel{I}{\to} PC^{*,*}\to CC^{*,*}[-1,-1]\to 0
\end{equation}
For instance, one has $NC^{-1,-1}=0$, $PC^{-1,-1}=CC^{0,0}$, $CC^{-1,-1}[-1,-1]=CC^{0,0}$. We let $ PC^j= {\rm Tot}^j(PC^{*,*})$ and $NC^j= {\rm Tot}^j(NC^{*,*})$ be the total complexes and we denote with $I:NC^j\to PC^j$ the canonical inclusion. Then \eqref{shortcomp} yields the short exact sequence of total complexes of cochains
\begin{equation}\label{shortcomp1}
    0\to NC^*\stackrel{I}{\to} PC^*\to CC^*[-2]\to 0.
\end{equation}
This sequence induces the long exact sequence of hypercohomology groups $\mathbb H^{-n}$
 \begin{equation}\label{sbibis}
    \cdots\to HN_n(X_k) \stackrel{I}{\to} HP_n(X_k)\stackrel{S}{\to}HC_{n-2}(X_k)\stackrel{B}{\to} HN_{n-1}(X_k)\to\cdots
\end{equation}
that sits on top of the $SBI$ long exact sequence in a canonical commutative diagram (\cf \cite{Loday}, Proposition 5.1.5).
The $\lambda$-decomposition is compatible with \eqref{sbibis}, therefore one also has the following exact sequence (\cf \cite{Loday}, Theorem 4.6.9, for the analogous statement for the SBI sequence)
\begin{equation*}\label{sbibisjj}
    \cdots\to HN^{(j)}_n(X_k) \stackrel{I}{\to} HP^{(j)}_n(X_k)\stackrel{S}{\to}HC^{(j-1)}_{n-2}(X_k)\stackrel{B}{\to} HN^{(j)}_{n-1}(X_k)\to\cdots
\end{equation*}
Note that the index $j$ is fixed all along the sequence.
The sequence \eqref{shortcomp1} determines a canonical isomorphism
\begin{equation}\label{coneident}
HC_{n-2}(X_k) \cong \mathbb H^{-n}(X_k,\text{Cone}(NC^*\stackrel{I}{\to}PC^*)).
\end{equation}

\subsection{The archimedean cyclic homology of $X_\C$}\label{sectarcdefn}
For a smooth complex projective variety $X_\C$ we consider the complex which is (up to a shift of degrees) the quasi-pullback of the two maps of complexes (\cf Definition~\ref{maptau})
\[
 I:NC^*(X_\C)\to PC^*(X_\C),\qquad{\rm and}\qquad \tau: PC_\rrr^*(X_\C)\to PC^*(X_\C).
 \]
\begin{defn}\label{realhar} The archimedean cyclic homology of a smooth projective complex variety $X_\C$  is defined as follows
\begin{equation*}\label{hardefngen}
    \har_n(X_\C):= \mathbb H^{-n}\left(X_\C,\text{Cone}\left( NC^*(X_\C)\oplus PC_\rrr^*(X_\C) \stackrel{\beta}{\to} PC^*(X_\C)                      \right)[2]\right)
\end{equation*}
where the map $\beta$ is given by
\begin{equation*}
    \beta(\omega,a)=I(\omega)-\tau(a).
\end{equation*}
\end{defn}

The long exact sequence of cohomology groups  associated to \eqref{basesequ} produces in this context the following
\begin{prop}\label{proplongexactseq}
There is a long exact sequence of the form
\begin{multline}\label{lloong}
   \cdots\to HP_{n+2}^\rrr(X_\C)\stackrel{S}{\to}HC_n(X_\C)\to \har_n(X_\C)\to
   HP_{n+1}^\rrr(X_\C)\stackrel{S}{\to}\\ \to HC_{n-1}(X_\C)\to \har_{n-1}(X_\C)\to \cdots
\end{multline}
\end{prop}
\proof The exact sequence of complexes \eqref{basesequ} takes here the  form
\begin{multline*}
    0\to \text{Cone}\left(NC^*(X_\C) \stackrel{I}{\to} PC^*(X_\C)\right)\to
   \text{Cone}\left(PC_\rrr^*(X_\C)\oplus NC^*(X_\C) \stackrel{\beta}{\to} PC^*(X_\C)\right)
    \to\\ \to PC_\rrr^*(X_\C)[-1]\to 0
\end{multline*}
Moreover, \eqref{coneident} reads as $HC_{n-2}(X_\C) = \mathbb H^{-n}(X_\C,\text{Cone}(NC^*(X_\C)\stackrel{I}{\to}PC^*(X_\C)))$.
Passing to hypercohomology  one gets the long exact sequence
\begin{equation*}\label{lexxx}
\cdots\to HC_{n-2}(X_\C)\to \har_{n-2}(X_\C)\to \mathbb H^{-n}(X_\C, PC_\rrr^*(X_\C)[-1])\to HC_{n-3}(X_\C)\to \cdots
\end{equation*}
which takes  the form
\begin{equation*}\label{lloongbis}
   \cdots \to HC_{n-2}(X_\C)\to \har_{n-2}(X_\C)\to
   HP_{n-1}^\rrr(X_\C)\to HC_{n-3}(X_\C)\to \cdots
\end{equation*}
\endproof
Next, we compare \eqref{lloong} with the  following exact sequence of real Deligne cohomology (\cf \S\ref{passingsect})
\begin{equation}\label{Ds}
 \to H^m_B(X(\C),\R(r))\to  H_{{\rm dR},\text{rel}}^{m}(X(\C),r)\to H^{m+1}_{\mathcal D}(X_\C,\R(r)) \to H^{m+1}_B(X(\C),\R(r))\to\cdots
\end{equation}
To achieve a correct comparison we implement in \eqref{lloong} the $\lambda$-decomposition and deduce  the following exact sequence
\begin{multline}\label{seqcrux}
   \to HP_{n+2}^{\rm real,\,(j+1)}(X_\C)\stackrel{S}{\to}HC_n^{(j)}(X_\C)\to {\har_n}^{,(j)}(X_\C)\to\\
\to   HP_{n+1}^{\rm real,\,(j+1)}(X_\C)\stackrel{S}{\to}HC_{n-1}^{(j)}(X_\C)\to \cdots
\end{multline}
Notice that in both exact sequences \eqref{Ds} and \eqref{seqcrux} the indices $j$ and $r$ are kept fixed, while the value of $m=2j-n$ increases by $1$ as $n$ decreases by $1$.
With $m=2j-n$ and $r=j+1$, one then argues the existence of the following correspondence:
\vspace{.1in}

\tracingtabularx
\hglue 23mm\begin{tabularx}{\linewidth}%
{>{\setlength{\hsize}{.8\hsize}}X|%
>{\setlength{\hsize}{.7\hsize}}X}
$H_B^{2(j+1)-(n+2)}(X(\C),\R(j+1))$ & $HP_{n+2}^{\rm real,\,(j+1)}(X_\C)$\\
$\tilde H^{2j-n+1}_{\mathcal D}(X_\C,\R(j+1))=H_{{\rm dR},\text{rel}}^{2j-n}(X(\C),j+1)$ & $HC_n^{(j)}(X_\C)$\\
$H^{2j-n+1}_{\mathcal D}(X_\C,\R(j+1))$& ${\har_n}^{,(j)}(X_\C)$\\
\end{tabularx}
\vspace{.1in}

If $X_k$ is a smooth, projective variety over a field $k$ of characteristic zero (so in particular in the case $k=\C$) the ``Hodge-to-de Rham'' spectral sequence degenerates, then by applying Proposition 4.1 (b) in \cite{Weibelcris}, \eqref{sbibis} brakes into  short exact sequences
\begin{equation}\label{sbibis1}
 0   \to HN_{n+2}(X_k) \stackrel{I}{\to} HP_{n+2}(X_k)\stackrel{S}{\to}HC_{n}(X_k)\to 0.
\end{equation}
By implementing the $\lambda$-decomposition, for each $j\geq 0$, one derives exact sequences of the form
\begin{equation}\label{sbibis2}
 0   \to HN_{n+2}^{(j+1)}(X_k) \stackrel{I}{\to} HP_{n+2}^{(j+1)}(X_k)\stackrel{S}{\to}HC_{n}^{(j)}(X_k)\to 0
\end{equation}
In view of Proposition 4.1 (a) of \opcit and the isomorphism \eqref{cyclicequ1} of Proposition \ref{propcyclic1} \ie \begin{equation*}\label{sbibis4}
HC_{n}^{(j)}(X_\C)\cong H_{{\rm dR},\text{rel}}^m(X(\C),j+1)
\end{equation*}
the sequences \eqref{sbibis2} correspond, for $m=2j-n$ and $k=\C$, precisely to the exact sequences
\begin{equation*}\label{sbibis3}
 0   \to F^{j+1} H_{{\rm dR}}^m(X_{\C}) \stackrel{\iota}{\to} H_{{\rm dR}}^m(X_{\C})\to H_{{\rm dR},\text{rel}}^m(X_\C,j+1)\to 0.
\end{equation*}
The following result and the next Corollary~\ref{mainlemC} validates the above argued correspondence
\begin{prop}\label{agree} For a smooth, complex projective variety $X_\C$ and for any pair of integers $(n,j)\in E_d$, there is a short exact sequence
\begin{equation}\label{cp}
   0\to HP_{n+2}^{\rm real,(j+1)}(X_\C)\stackrel{S\circ \tau }{\longrightarrow}
   HC_n^{(j)}(X_\C)\to {\har_n}^{,(j)}(X_\C)\to 0.
\end{equation}
\end{prop}
\proof
Let $(n,j)\in E_d$, then by Lemma \ref{trivchange} one derives that $(2j-n,j-n)\in A_d$. One then has the short exact sequence \eqref{seqshort} (for $w=2j-n$ and $r=j+1$)
\begin{equation}\label{seqshort2}
 0\to H^{2j-n}_B(X(\C),\R(j+1))\to  H_{{\rm dR},\text{rel}}^{2j-n}(X_\C,j+1)\to H^{2j+1-n}_{\mathcal D}(X_\C,\R(j+1)) \to 0.
\end{equation}
Moreover, one has a canonical isomorphism (\cf \eqref{cyclicequ1} of Proposition \ref{propcyclic1})
\begin{equation*}\label{cyclicequ1bis}
    HC_n^{(j)}(X_\C)\cong H_{{\rm dR},\text{rel}}^{2j-n}(X_\C,j+1)
\end{equation*}
One then uses \eqref{Smap} and Proposition \ref{realiden} to identify the map $HP_{n+2}^{\rm real,(j+1)}(X_\C)\stackrel{S\circ \tau }{\longrightarrow}
   HC_n^{(j)}(X_\C)$ with the map
\begin{equation*}
H^{2j-n}_B(X(\C),\R(j+1))\stackrel{\subset}{\to}  H_{{\rm dR},\text{rel}}^{2j-n}(X_\C,j+1).
\end{equation*}
\endproof
\begin{cor}\label{mainlemC}
Let $X_\C$ be a smooth, complex projective variety of dimension $d$. Then for any pair of integers $(n,j)\in E_d$ one has
\begin{equation*}\label{mainequc}
    H^{2j+1-n}_{\cD}(X_\C,\R(j+1))\cong\har_n(X_\C)^{\Theta_0=j}.
\end{equation*}
Moreover $\har_n(X_\C)^{\Theta_0=j}=\{0\}$ for
$n\geq 0,\,j\geq 0$ and $(n,j)\notin E_d$.
\end{cor}
\proof The first statement follows by comparing the two exact sequences \eqref{seqshort2} and
\eqref{cp}. We prove the second statement.
By applying Proposition 3.1 of \cite{Weibelcris}, one derives for $j<\frac n2$ that $HC_n^{(j)}(X_\C)=0$. Moreover the same vanishing also holds for $2j-n>2d$ (\opcit Proposition 4.1). Proposition \ref{realiden} shows that for $2j-n>2d$ one has (since $2(j+1)-(n+2)>2d$ and $2(j+1)-(n+1)>2d$)
\begin{equation}\label{complvanish}
    HP_{n+2}^{\rm real,(j+1)}(X_\C)=0, \ \ HP_{n+1}^{\rm real,(j+1)}(X_\C)=0.
\end{equation}
Thus the exact sequence \eqref{seqcrux} shows that $\har_n(X_\C)^{\Theta_0=j}=\{0\}$ for
$2j-n>2d$. For $j<\frac n2$, if  $j<\frac{ n-1}{2}$
then Proposition \ref{realiden} shows that \eqref{complvanish} holds. Let us consider now the limit case $n=2j+1$. In this case one has $2(j+1)-(n+1)=0$ and thus $
    HP_{n+1}^{\rm real,(j+1)}(X_\C)=H_B^0(X(\C),\R(j+1))\neq 0$
but the next map in  \eqref{seqcrux}, is the same as the map
\begin{equation*}
    H_B^0(X(\C),\R(j+1))\to H_{{\rm dR},\text{rel}}^{2j-(n-1)}(X_\C,j+1)\cong HC_{n-1}^{(j)}(X_\C)
\end{equation*}
and this homomorphism is injective since $H_B^0\cap F^{j+1}=0$. Thus \eqref{seqcrux} shows that $\har_n(X_\C)^{\Theta_0=j}=\{0\}$.
\endproof

\subsection{The case of real schemes}\label{sectreal}

 Next, we consider the case of a smooth, complex projective variety  defined over $\R$, $X_\R$. In this case following \cite{Weibelcris} $X_\R$ admits a cyclic homology theory over $\R$.
\begin{lem}\label{extscal}
Let $X_\R$ be a smooth, real projective variety, and denote by  $X_\C=X_\R\otimes_\R\C$ the variety obtained by extension of scalars. Then
\begin{equation}\label{extscal1}
    HC_n(X_\C)=HC_n(X_\R)\otimes_\R\C \qqq n\in\Z.
\end{equation}
\end{lem}
\proof The result follows from Proposition 4.1 of \cite{Weibelcris} and the well known decomposition
\begin{equation}\label{drdrbis}
  H^q(X_\C,\Omega_{X_\C}^p)=  H^q(X_\R,\Omega_{X_\R}^p)\otimes_\R\C.
\end{equation}
\endproof
From Lemma \ref{extscal} one derives the definition of a canonical anti-linear involution $\bar F_\infty={\rm id}\otimes \bar{~}$ on $HC_n(X_\C)$, by implementing the complex conjugation $\bar{~}$ on $\C$. This involution  should not be confused with the $\C$-linear involution $F_\infty$ used in \cite{Deninger91}. By Proposition 1.4 of \cite{De5}, the corresponding anti-linear involution $\bar F_\infty$ acting on the Betti cohomology $H^*(X_\R,\C) = H^*(X_\R,\R)\otimes_\R\C$  is  $F_\infty\otimes\bar{~}$. Note that $\bar F_\infty$ preserves the Hodge spaces $H^q(X_\R,\Omega_{X_\R}^p)$.\newline
Proposition \ref{realiden} implies that the subspace  of $HC_n^{(j)}(X_\R)$ which is the image of $HP_{n+2}^{\rm real,(j+1)}(X_\R)$ through the composite $S\circ \tau$ is globally invariant under $\bar F_\infty$ thus the latter operator descends on the quotient.

We can now introduce the  definition of the archimedean cyclic homology for real varieties.
\begin{defn}\label{realvar}
Let $X_\R$ be a smooth real projective variety and let  $X_\C=X\otimes_\R\C$ be the associated complex variety. Define the archimedean cyclic homology of $X_\R$ as
\begin{equation}\label{extscal1}
    \har_n(X_\R):=\har_n(X_\C)^{\bar F_\infty={\rm id}}.
\end{equation}
\end{defn}
Note that one can re-write \eqref{extscal1} in the equivalent form
\begin{equation}\label{extscal2}
    \har_n(X_\R)=HC_n(X_\R)/\left( HC_n(X_\R)\cap {\rm Im}(S\circ \tau)   \right)
\end{equation}
since it is easy to see that the natural map from the right hand side of \eqref{extscal2} to the left side is bijective.
\begin{cor}\label{mainlem}
Let $X_\R$ be a smooth real projective variety of dimension $d$. Then for a pair of integers $(n,j)\in E_d$ we have
\begin{equation}\label{mainequ}
    H^{2j+1-n}_{\cD}(X_{\R},\R(j+1))=\left(\har_n(X_{\R})\right)^{\Theta_0=j}.
\end{equation}
Moreover $\har_n(X_{\R})^{\Theta_0=j}=\{0\}$ for
$n\geq 0,\,j\geq 0$ and $(n,j)\notin E_d$.
\end{cor}
\proof The proof is easily adapted from that of Corollary \ref{mainlemC}.\endproof

Let us consider the simplest example of $X_\R=\Spec \R$.  In this case the exact sequence \eqref{cp} gives for $X_\C=\Spec\C$
\begin{equation*}\label{cp1}
   0\to HP_{n+2}^{\rrr}(\Spec\C)\stackrel{S\circ \tau}{\longrightarrow}
   HC_n(\Spec\C)\to \har_n(\Spec\C)\to 0.
\end{equation*}
One therefore derives
\begin{equation*}\label{cp2}
    \har_{2m}(\Spec\C)=\C/\left((2\pi i)^{m+1}\R\right)
\end{equation*}
while the odd homology groups all vanish. This is in agreement with the well-known result for real Deligne cohomology. For $X_\R=\Spec \R$, the de Rham conjugation $\bar F_\infty$ is the ordinary complex conjugation, and we claim that it admits a non-zero fixed vector in $ \har_{2m}(\Spec\C)$ if and only if $m$ is even. Indeed, in that case the quotient $\C/\left((2\pi i)^{m+1}\R\right)$ coincides with $\R$, while it reduces to $i\R$ when $m$ is odd. Thus we derive
\begin{equation*}\label{cp3}
    \har_{4m}(\Spec\R)=\C/\left((2\pi i)^{2m+1}\R\right), \ \ \har_{n}(\Spec\R)=\{0\}, \quad \text{for}~ n\not\equiv 0\, (4).
\end{equation*}
  It follows that the operator $\Theta=\Theta_0-\Gamma$ acts on $\har_{4m}(\Spec\R)$ by multiplication by $2m-4m$ and hence its spectrum consists of all negative  even  integers $2k\leq 0$. One then derives  the formula for the local (archimedean) Euler factor
 \begin{equation*}\label{caser}
   \Gamma_\R(s)^{-1}=\prod_{k\leq 0}\left(\frac{s-2k}{2\pi}\right).
 \end{equation*}

\begin{rem}\label{denig}{\rm  The archimedean cohomology defined in \cite{Deninger91} differs from the real Deligne cohomology in the case of real places and when the weight $w$ is odd. Indeed, by \cite{Deninger91} (\cf proof of Proposition 5.1) one gets in such case (we use the notation of \opcit)
\begin{equation*}
    \left(H_{\rm ar}^w\right)^{\Theta=m}=
    H^{w+1}_{\cD}(X_{/\C},\R(n))^{\bar F_\infty=-id}
\end{equation*}
The real archimedean cyclic homology on the other hand does not have this mismatch and it produces exactly the real Deligne cohomology.
}
\end{rem}

\section{Proof of Theorem \ref{MT}}\label{sectmain}

\begin{thm}\label{main} Let $X$ be a smooth, projective variety of dimension $d$ over an algebraic number field $K$ and let $\nu\vert\infty$ (\ie $K_\nu = \C,\R$) be an archimedean place of $K$.
Then the action of the operator $\Theta=\Theta_0-\Gamma$ on the archimedean cyclic homology of $X_\nu$ satisfies the following formula
\begin{equation}\label{dettheta}
    \prod_{0\leq w \leq 2d} L_\nu(H^w(X),s)^{(-1)^{w+1}}=\frac{det_\infty(\frac{1}{2\pi}(s-\Theta))|_{\har_{\rm even}(X_\nu)}}{
    det_\infty(\frac{1}{2\pi}(s-\Theta))|_{\har_{\rm odd}(X_\nu)}},\qquad s\in\R.
\end{equation}
The left-hand side of the formula describes the product of Serre's archimedean local factors in the complex $L$-function of $X$ and $det_\infty$ denotes the regularized determinant (\cf\eg\cite{RS},\cite{Deninger91}).
\end{thm}
\proof By \cite{Deninger91}, Proposition 2.1, it is enough to show that both sides of \eqref{dettheta} have the same divisor of zeros and poles. We fix a Hodge weight $w$ and consider the archimedean local factor $L_\nu(H^w(X),s)$ associated first to a real place $\nu$ of $K$ (\ie $K_\nu = \R$). This function is specified by the multiplicities of its poles and it is well known that these latter only occur at integer values $s=m\leq \frac w2$ and their multiplicity is described by the following formula
\begin{equation}\label{ordpolebis}
    {\rm ord}_{s=m}L_\nu(H^w(X),s)^{-1}={\rm dim}_\R H^{w+1}_{\cD}(X_\nu,\R(w+1-m)).
\end{equation}
Here $X_\nu=X_{/\R}$ and the corresponding real Deligne cohomology groups are defined in \eqref{realrealdefn}.
By Lemma \ref{trivchange}, the pairs $(w,m)$ which enter in \eqref{ordpolebis} correspond bijectively to pairs $(n,j)\in E_d$, for $w=2j-n$, $m=j-n$. Using the identification obtained in Corollary \ref{mainlem} we also know that
\begin{equation}\label{mainequ}
    H^{2j+1-n}_{\cD}(X_{\R},\R(j+1))=\left(\har_n(X_{\R})\right)^{\Theta_0=j}\qqq (n,j)\in E_d.
\end{equation}
Hence by combining \eqref{ordpolebis} with  \eqref{mainequ} we get
\begin{equation*}\label{ordpole1}
    {\rm ord}_{s=m}L_\nu(H^w(X),s)^{-1}={\rm dim}_\R \left(\har_n(X_{\R})\right)^{\Theta_0=j}.
\end{equation*}
Thus, at a real place $\nu$, the divisor $\Delta$ of the function
\begin{equation*}
    \zeta_\nu(X)= \prod_{0\leq w \leq 2d} L_\nu(H^w(X),s)^{(-1)^{w+1}}
\end{equation*}
is given by the formula
\begin{align}\label{delta}
    \Delta&=\sum_{(w,m)\in A_d}(-1)^w {\rm ord}_{s=m}L_\nu(H^w(X),s)^{-1}\delta
    =\sum_{(n,j)\in E_d}(-1)^n {\rm dim}_\R \har_n(X_{\R})^{\Theta_0=j}\delta_m\\
    &=\sum_{n,j\geq 0}(-1)^n {\rm dim}_\R \har_n(X_{\R})^{\Theta_0 =j }\delta_{j-n}.\notag
\end{align}
This because $w\equiv n$ (mod $2$) and $j-n=m$ and because $\har_n(X_{\R})^{\Theta_0=j}=\{0\}$ for
$n\geq 0,\,j\geq 0$ and $(n,j)\notin E_d$ (Lemma \ref{mainlem}). We rewrite the right hand side of \eqref{delta} as
\begin{equation*}
    \sum_{n,j\geq 0}(-1)^n {\rm dim}_\R \har_n(X_{\R})^{\Theta_0 =j }\delta_{j-n}
    =\sum_{n\geq 0,\, k\in \Z}(-1)^n {\rm dim}_\R \har_n(X_{\R})^{\Theta_0-\Gamma =k }\delta_{k}
\end{equation*}
This in turns gives, using the operator $\Theta= \Theta_0-\Gamma$
\begin{equation*}\label{main2}
     \Delta=\sum_{k\in \Z} {\rm dim}_\R \har_{\rm even}(X_{\R})^{\Theta =k }\delta_{k}-\sum_{k\in \Z} {\rm dim}_\R \har_{\rm odd}(X_{\R})^{\Theta =k }\delta_{k}
\end{equation*}
which represents the signed spectral measure of the operator $\Theta$ acting on $\har_*(X_\R)$ with the convention that the index $n$ in $\har_n(X_\R)$ is always restricted to be $n\geq 0$.

For a complex place one proceeds in the same way, using Corollary \ref{mainlemC} instead of Corollary \ref{mainlem}.
\endproof

\begin{rem}\label{more}{\rm The operators $\Theta_0$ and $\Gamma$ commute and the proof of Theorem \ref{main} shows that the decomposition of the alternate product on the left hand side of equation \eqref{dettheta} into factors associated to specific weights corresponds to the spectral decomposition of the operator $2\Theta_0-\Gamma$, which commutes with $\Theta$. Thus, by restricting the right hand side of \eqref{dettheta} to the spectral projection $2\Theta_0-\Gamma=w$, one recovers precisely the factor $L_\nu(H^w(X),s)^{(-1)^{w+1}}$.
In this way we refine \eqref{dettheta}.
}\end{rem}

\section{Appendix: Weil restriction and the  functor $X\mapsto {}_*X$}\label{sectweil}

Let $X$ be a smooth, projective complex variety viewed as a scheme and $X_\aaa$ its set of closed points.
 The results of Section \ref{sectrealcyc} are based on the canonical morphism of locally ringed spaces
$\mu:X_\sss\to X_\aaa$ that links $X_\aaa$ to the associated $C^\infty$-manifold $X_\sss$, viewed as the locally ringed space with structure sheaf $\cO_{X_\sss}(U):=C^\infty(U,\C)$ for every open subset $U\subset X_\sss$. Although $\mu$ is very useful and plays an important role in our construction it does not belong properly to the realm of algebraic geometry. In particular the locally ringed space $X_\sss$ is not a scheme and differs substantially from the affine scheme $\Spec(A)$, with $A=\cO(X_\sss)=C^\infty(X_\sss,\C)$. In this section we show how one may {\em extend} $\mu$ to a {\em morphism of schemes}
\begin{equation*}
    \pi_X:\Spec(A)\to X.
\end{equation*}

In fact, we shall construct a pair made of a natural additive functor $X\mapsto {}_*X$ from the category of schemes over $\C$ to the category of schemes over $\R$ and a natural transformation of functors
\begin{equation*}
    p_X: ({}_*X)_\C\to X, \ \ ({}_*X)_\C:= {}_*X\otimes_\R \C
\end{equation*}
which fulfills the following properties when applied to projective varieties  over $\C$:

$(i)$~The {\em real} scheme ${}_*X$ is affine.\newline
$(ii)$~The morphism $p_X$ is bijective on complex points, \ie $p_X$ induces a bijection ${}_*X(\C)\cong X(\C)$.\newline
$(iii)$~If $X$ is also smooth, the morphism of locally ringed spaces $X_\sss\to X_\aaa$ {\em extends canonically} to a morphism of schemes
 \begin{equation*}
 \pi_X:\Spec(\cO(X_\sss))\stackrel{\gamma^*}{\to} ({}_*X)_\C \stackrel{p_X}{\to} X
\end{equation*}
 where $\gamma: \cO({}_*X)\to \cO(X_\sss)=C^\infty(X_\sss,\C)$ is  the  Gelfand transform of the algebra
 $B=\cO({}_*X)$ of global regular sections of the affine scheme ${}_*X=\Sp(B)$:
\begin{equation*}
    B=\cO({}_*X)\ni a\mapsto \gamma(a)=\hat a, \ \ \hat a(\chi):=\chi(a)\qqq \chi \in \Hom_\C(B,\C).
\end{equation*}

\subsection{The functor $X\mapsto {}_*X$}

We start by recalling the following well known construction in algebraic geometry. Let $Z$ be a scheme and let $\cL$ be an invertible sheaf on $Z$. Given a global section $s\in \cL(Z)$ of $\cL$, one defines the following open subset $Z_s\subset Z$
\begin{equation*}\label{zs}
    Z_s:=\{z\in Z\mid \cL_z=s_z\cO_{Z,z}\}.
\end{equation*}

It is a standard fact that the inclusion $Z_s \hookrightarrow Z$ is an affine morphism, thus if $Z$ is an affine scheme then $Z_s$ is an affine open subset of $Z$ (\cf\cite{GW} Example 12.4(4)).

Now, let $X_{}$ be a scheme over $\C$ and let $Y_\R={\rm Res}_{\C/\R}X_{}$ be the Weil restriction. Then define
\begin{equation*}
   Z_{}:= Y_\R\times_\R\C=\prod_{\rho}X_{}^\rho
\end{equation*}
where the product of schemes  is over $\C$, $\rho$ varies between the two $\R$-linear embeddings $\C\to \C$ (\ie the identity and the complex conjugation $\sigma$) and $X_{}^\rho$ is the complex scheme obtained from $X_{}$ by extension of scalars using $\rho:\C\to \C$. The two projections $p_\rho: Z_{}\to X_{}^\rho$ are both morphisms of schemes. The action of $\Gal(\C/\R)$ on $Z_{}$ determines a canonical isomorphism $\tilde \sigma: Z\stackrel{\sim}{\to} Z^\sigma$.

Let $\cL$ be an invertible sheaf on $X_{}$ and $\cL^\sigma$ the corresponding invertible sheaf over $X_{}^\sigma$. Furthermore, let $\tilde \cL$ be the invertible sheaf on  $Z_{}$
 \begin{equation*}\label{tildeL}
    \tilde \cL:=p^*_{{\rm id}}\cL\otimes p_\sigma^*\cL^\sigma.
\end{equation*}
The action of $\Gal(\C/\R)$ on $Z_{}$ determines a canonical  anti-linear involution on sections of $\tilde \cL$ on any $\tilde \sigma$-invariant open set $U\subset Z$,
\begin{equation*}\label{invol}
     \Gamma(U,\tilde \cL)\ni \xi\mapsto \xi^* \in \Gamma(U,\tilde \cL)
\end{equation*}
such that, in particular
\begin{equation*}\label{invol1}
  (\xi\otimes \eta^\sigma)^*=\eta\otimes \xi^\sigma  \qqq\xi, \eta \in \Gamma(U,p^*_{{\rm id}}\cL).
\end{equation*}
Let $\{\xi_j\}_{j\in J}\in \Gamma(Z_{},\tilde \cL)$ be a family of global sections.
 Consider the open subset of $Z_{}$
 \begin{equation}\label{invol2}
    \Omega(\cL, \{\xi_j\})=Z_s, \qquad s:=\sum_{j\in J} \xi_j\xi_j^*\in \Gamma(Z_{},\tilde\cL^{\otimes 2}).
\end{equation}
The subsets $\Omega(\cL, \{\xi_j\})$ satisfy the following  property
\begin{lem} \label{yprod} Let  $\{\xi_j\}_{j\in J}$ and $\{\eta_k\}_{k\in I}$ be families of global sections of two invertible sheaves $\cL_1$ and $\cL_2$ on $Z_{}$, then
\[
\Omega(\cL_1\otimes \cL_2, \{\xi_j\otimes \eta_k\})=\Omega(\cL_1, \{\xi_j\})\cap \Omega(\cL_2, \{\eta_k\}).
\]
\end{lem}
\proof By construction one has $\widetilde{(\cL_1\otimes \cL_2)}=\tilde \cL_1\otimes \tilde \cL_2$ so that $\xi_j\otimes \eta_k\in \Gamma(Z_{},\widetilde{(\cL_1\otimes \cL_2)})$. Moreover the involution $*$ on  $\Gamma(Z_{},\widetilde{(\cL_1\otimes \cL_2)})$ is the tensor product of the two involutions.
Thus the sections in \eqref{invol2} are given by
\begin{equation*}
    s=\sum_{j,k} (\xi_j\otimes \eta_k)(\xi_j\otimes \eta_k)^*
    =\left(\sum_j \xi_j \xi_j^*\right)\otimes \left(\sum_k \eta_k \eta_k^*\right)=s_1\otimes s_2.
\end{equation*}
It follows that $\Omega(\cL_1\otimes \cL_2, \{\xi_j\otimes \eta_k\})=Z_s=Z_{s_1}\cap Z_{s_2}=\Omega(\cL_1, \{\xi_j\})\cap \Omega(\cL_2, \{\eta_k\})$.\endproof

\begin{defn}\label{starx} Let $X_{}$ be  a scheme over $\C$. We denote by $S(X_{})$ the smallest class of open subsets of $Z_{}=X_{}\times_\C X_{}^\sigma$ such that $Z_{}\in S(X_{})$ and  for any invertible sheaf $\cL$ on $X_{}$ and a family of global sections $\{\xi_j\}_{j\in J}\in \Gamma(Z_{},\tilde\cL)$ the following condition holds
\begin{equation}\label{induct}
  \exists W\in S(X_{}), \ W\subset \bigcup_j {Z_{}}_{\xi_j}\implies  \Omega(\cL, \{\xi_j\})\in S(X_{}).
\end{equation}
\end{defn}
One has $S(X_{})=\cup_{n\ge 0} S_n(X_{})$, where $S_n(X_{})$ is defined by induction: one takes $S_0(X_{})=\{Z_{}\}$ and defines
\begin{equation*}
    S_{n+1}(X_{})=\{\Omega(\cL, \{\xi_j\})\mid \exists  W\in S_n(X_{}), \ W\subset
    \bigcup_j {Z_{}}_{\xi_j}\}.
\end{equation*}
 If $\cL=\cO_{X_{}}$ and the family $\{\xi_j\}$ reduces to the constant section $\{1\}$, one gets $S_0(X_{})\subset S_1(X_{})$ and by induction, $S_n(X_{})\subset S_{n+1}(X_{})$.

\begin{lem} \label{yincl}
$(i)$~The class $S(X_{})$ is stable under finite intersections.

 $(ii)$~For any $W\in S(X_{})$:~
$
Y(\R)\subset  W(\C)
$.

$(iii)$~Any $W\in S(X_{})$ fulfills $\tilde \sigma (W)=W$.
\end{lem}
\proof $(i)$~We show inductively that $S_n(X_{})$ is stable under finite intersections. The statement holds trivially for $n=0$. We assume $(i)$ holds for $S_n(X_{})$. Let $\Omega(\cL_1, \{\xi_j\})$,  $\Omega(\cL_2, \{\eta_k\})$
belong to $S_{n+1}(X_{})$. Then there exist $W_i\in S_n(X_{})$ ($i=1,2$) such that $W_1\subset \cup {Z_{}}_{\xi_j}$, $W_2\subset \cup {Z_{}}_{\eta_k}$.
Then, using the induction hypothesis  it follows that
\begin{equation*}
    W_1\cap W_2\in S_n(X_{}), \ \ W_1\cap W_2\subset \cup_{j,k} ({Z_{}}_{\xi_j}\cap {Z_{}}_{\eta_k})
    =\cup_{j,k} {Z_{}}_{\xi_j\otimes \eta_k}.
\end{equation*}
Then Lemma \ref{yprod} shows that $\Omega(\cL_1, \{\xi_j\})\cap \Omega(\cL_2, \{\eta_k\})$
belongs to $S_{n+1}(X_{})$.

$(ii)$~We again proceed by induction on $n$ for $W\in S_n(X_{})$. The case $n=0$ is clear
since by functoriality one has $Y(\R)\subset Y(\C)$. Assume to have shown that
  $
Y(\R)\subset  W(\C)
$ for any $W\in S_n(X_{})$, then we prove the statement for $\Omega(\cL, \{\xi_j\})\in S_{n+1}(X_{})$. Let $x\in Y(\R)$, then, since $Y(\R)=X(\C)$ there exists a morphism
$\chi\in \Hom(\Spec\C,X_{})$ such that the corresponding point of $Y(\C)=X(\C)\times X_{}^\sigma(\C)$ is given by $\chi\times (\chi\circ \sigma)$. By hypothesis, there exists $W\in S_n(X_{})$ such that $W\subset \cup {Z_{}}_{\xi_j}$ and so the induction hypothesis implies that one has
$\chi\times (\chi\circ \sigma)\in {Z_{}}_{\xi_{j_0}}(\C)$ for some index $j_0$. Pulling back the invertible sheaf $\cL$ by $\chi$ one gets a rank one $\C$-vector space $L$. Let $0\neq e\in L\otimes \bar L$ be a basis, then the pull back of the sections $\xi_j\in \Gamma(Z_{},\tilde\cL)$ are of the form $a_j e$ with $a_j\in \C$ and $a_{j_0}\neq 0$. It follows that
\begin{equation*}
    \left(\chi\times (\chi\circ \sigma)\right)^*\left(\sum_j \xi_j  \xi_j^*\right)
    =\left(\sum_j a_j \bar a_j\right)e e^*\neq 0
\end{equation*}
and thus $\chi\times (\chi\circ \sigma)\in \Omega(\cL, \{\xi_j\})(\C)$.

$(iii)$~This invariance under $\tilde \sigma$ holds for the subsets $\Omega(\cL, \{\xi_j\})$ because the section $s$ of \eqref{invol2} fulfills $s^*=s$ by construction.
 \endproof

Next statement provides the functorial construction of the (real: \cf Theorem \ref{affscheme}) scheme ${}_*X$ and the projection $p_X:({}_*X)_\C\to X_{}$.
\begin{prop}\label{functstar}
Let $X_{}$ be a scheme over $\C$. Then, the restriction of the first projection $p:X_{}\times_\C X_{}^\sigma\to X_{}$ to the intersection
\begin{equation*}\label{starXX}
    {}_*X:=\projlim_{W\in S(X_{})} W
\end{equation*}
defines a morphism $p_X:{}_*X\to X_{}$ of schemes which is surjective on complex points and depends functorially on $X_{}$. Moreover the functor $X_{}\mapsto {}_*X$ preserves direct sum, \ie
\begin{equation*}
    {}_*(X_{1}\oplus X_{2})={}_*X_{1}\oplus {}_*X_{2}.
\end{equation*}
\end{prop}
\proof It follows from Lemma \ref{yincl} that the class $S(X_{})$ is stable under finite intersections and thus the open (complex) sub-schemes $W$ of $X_{}\times_\C X_{}^\sigma$ form a projective system under inclusion. The projection $p$ restricts to each $W$ in a compatible manner, and Lemma \ref{yincl} implies that it is surjective on complex points. We check that this construction is functorial. Let $f:X_{}\to X'_{}$ be a morphism of complex schemes. Then the morphism $f\times f^\sigma:  X_{}\times_\C X_{}^\sigma\to X_{}'\times_\C {X'_{}}^\sigma$
is compatible with the first projection and the morphism $f:X_{}\to X'_{}$. It remains to show that for any $W'\in S(X'_{})$ there exists $W\in S(X_{})$ such that $(f\times f^\sigma)(W)\subset W'$.
For any invertible sheaf $\cL'$ on $X'_{}$ and sections $\xi'_j\in \Gamma(Z'_{},\tilde\cL')$, one has:~
$\left(f\times f^\sigma\right)^{-1} \Omega(\cL', \{\xi'_j\})=\Omega(\cL, \{\xi_j\})$,
where $\cL=f^*(\cL')$ is the inverse image of $\cL'$ and $\xi_j=\left(f\times f^\sigma\right)^*\xi'_j$. One then shows by induction on $n$ that
\begin{equation*}
    W'\in S_n(X'_{})\implies \left(f\times f^\sigma\right)^{-1}(W')\in S_n(X_{}).
\end{equation*}
Now let $X_{}=X_{1}\oplus X_{2}$ be the disjoint union of two schemes over $\C$. Let $e_j\in \cO(X_{})$ ($j=1,2$) be the two idempotent sections corresponding to $X_{j}$. One has $e_1+e_2=1$ and $e_1e_2=0$.  Consider the invertible sheaf $\cL=\cO_{X_{}}$ on $X_{}$. Then the tensor products $\xi_{ij}=e_i\otimes e_j^\sigma$ are sections of $\tilde \cL$ on $Z_{}=X_{}\times_\C X_{}^\sigma$. One has $\xi_{12}\xi_{12}^*=0$, $\xi_{21}\xi_{21}^*=0$  and
\begin{equation*}
    \cup_{i,j} Z_{\xi_{ij}}=Z,\ \ \Omega(\cL, \{\xi_{ij}\})=\left(X_{1}\times X_{1}^\sigma \right)\cup
   \left( X_{2}\times X_{2}^\sigma\right)\subset Z_{}.
\end{equation*}
One then easily checks that ${}_*X={}_*X_1\oplus {}_*X_2$.
\endproof
The additivity of the functor $X_{}\mapsto {}_*X$  holds because the star-localization eliminates the cross terms. Note that additivity fails instead for the Weil restriction.

\subsection{The morphism $\pi_X:\Spec(\cO(X_\sss))\to X$}\label{sectmorphis}
Next, we show that the functor $X\mapsto {}_*X$ transforms projective complex schemes  into {\em real affine} schemes
\begin{thm}\label{affscheme}
Let $X$ be a projective complex variety.  Then

$(i)$~The scheme ${}_*X$ is affine and real.

$(ii)$~The morphism $p_X: ({}_*X)_\C\to X$ is bijective on complex points, \ie $p_X$ induces a bijection ${}_*X(\C)\cong X(\C)$

$(iii)$~If $X$ is also smooth, the morphism of locally ringed space $X_\sss\to X_\aaa$ extends canonically to a morphism of schemes
 \begin{equation*}
 \pi_X:\Spec(\cO(X_\sss))\stackrel{\gamma^*}{\to} ({}_*X)_\C \stackrel{p_X}{\to} X
\end{equation*}
 where $\gamma: \cO({}_*X)\to \cO(X_\sss)=C^\infty(X_\sss,\C)$ is  the  Gelfand transform of the algebra
 $\cO({}_*X)$ of global regular sections of the affine scheme ${}_*X$:
\begin{equation*}
    \cO({}_*X)\ni a\mapsto \gamma(a)=\hat a, \ \ \hat a(\chi):=\chi(a)\qqq \chi \in \Hom_\C(\cO({}_*X),\C).
\end{equation*}
\end{thm}
\begin{proof}
 $(i)$~We show that the scheme ${}_*X$ is affine if $X$ is projective. Let $\iota:X\to \P^n_{\C}$ be a closed immersion and $\cL=\iota^*(\cO_{\P^n}(1))$. Let $\{\alpha_i\}\in \Gamma(X,\cL)$ be the set of pullback sections $\alpha_i=\iota^*(x_i)$ of the global sections $x_i\in\Gamma(\P^n,\cO(1))$, $i=0,\ldots, n$. Let $\xi_{ij}=\alpha_i\otimes \alpha_j^\sigma\in \Gamma(Z,\tilde\cL)$.
 One has $\cup_{i,j} Z_{\xi_{ij}}=Z$, so that $W_0=\Omega(\cL, \{\xi_{ij}\})\in S(X)$. Moreover,
  by definition one has
\begin{equation*}
\Omega(\cL, \{\xi_{ij}\})=Z_s,\qquad  s:=\sum_{i,j} \xi_{ij}\, \xi_{ij}^*\in \Gamma(Z,\tilde\cL^{\otimes 2}).
\end{equation*}
 In fact, one has
 \begin{equation*}
    s=\sum_{i,j} (\alpha_i\otimes \alpha_j^\sigma)(\alpha_j\otimes \alpha_i^\sigma)
    =\sum_{i,j} (\alpha_i\alpha_j)\otimes (\alpha_j^\sigma\alpha_i^\sigma)
    =\left(\sum_i \alpha_i\otimes \alpha_i^\sigma\right)
    \left(\sum_j \alpha_j\otimes \alpha_j^\sigma\right)=t^2
 \end{equation*}
 where $t=\sum_i \alpha_i\otimes \alpha_i^\sigma\in \Gamma(Z,\tilde\cL)$ and thus $W_0=\Omega(\cL, \{\xi_{ij}\})=Z_s=Z_t$. Next we show that $Z_t$ is affine.  Consider the Segre embedding $\tau: \P_\C^n\times \P_\C^n\to \P_\C^{n^2+2n}$
 which is given in terms of homogeneous coordinates by
 \begin{equation*}
    \tau((x_i), (y_j))= (x_iy_j)=\tau(x,y)_{i,j}, \ \ i,j\in \{0,\ldots, n\}.
 \end{equation*}
 The open set $Z_t\subset Z$ is, by construction, the inverse image by the morphism $\tau\circ(\iota\times \iota^\sigma)$ of the complement of the hyperplane
 \begin{equation*}
    H=\{(X_{i,j})\mid \sum_i X_{i,i}=0\}\subset \P_\C^{n^2+2n}
 \end{equation*}
 thus $W_0=\Omega(\cL, \{\xi_j\})=Z_s=Z_t$ is affine.  To show that ${}_*X$ is affine it is enough to show that $W_0\cap W$ is affine for any $W\in S(X)$ but this follows from standard geometric facts since $W_0\cap W$ is of the form $(W_0)_{s'}$ where $s'$ is a section of an invertible
sheaf on $W_0$. Finally by $(iii)$ of Lemma \ref{yincl},  each of the $W\in S(X)$ is invariant under $\tilde \sigma$. This shows that the algebras $\cO(W)$ inherit a natural antilinear involution $*$ and so does their projective limit $\cO({}_*X)$. Thus this algebra is the complexification of the real algebra of fixed points of the involution, thus it follows  that the scheme ${}_*X$ is real.

 $(ii)$~We show that the morphism $p_X$ is bijective on complex points. We already know that it is surjective. Let us consider a complex point $u\in {}_*X(\C)$. Then one gets a point $u\in Z(\C)=X(\C)\times X^\sigma(\C)$ which belongs to $W(\C)$ for any $W\in S(X)$. Let $W_0$ be  as in the proof of $(i)$. The point $u$ determines a character $\chi\in \Hom_\C(B,\C)$ of the involutive algebra $B=\cO(Z_s)$ and it is enough to show that this character fulfills the identity
\begin{equation}\label{selfa}
   \chi(x^*)=\overline{\chi(x)}\qqq x\in B.
\end{equation}
This will in turn follow if one shows that
\begin{equation}\label{selfa1}
    \chi(1+\sum_j a_j a_j^*)\neq 0\qqq a_j\in B.
\end{equation}
The set of complex numbers $\chi(\epsilon+\sum_j a_j a_j^*)$, $\epsilon>0$, $a_j\in B$, is a multiplicative cone, thus it reduces to $(0,\infty)$ unless it contains $0$. Moreover one has
 \begin{equation*}
    \chi(1+x x^*)\in (0,\infty)~~\forall x\in B\implies \chi(x^*)=\overline{\chi(x)}~~\forall x\in B.
 \end{equation*}
 To show \eqref{selfa1} it is enough to prove that with $f=1+\sum_j a_j a_j^*\in B$, one has $Z_f\in S(X)$.
Lemma 1.25 of \cite{Liu} applied to this context shows  that for any $a\in B$ there exists an integer $n_0>0$ such that for all $n\geq n_0$ the section $a\otimes (s^n_{|Z_s})\in \Gamma(Z_s,\tilde\cL^{\otimes 2n})$ extends to a global section $\widetilde{a s^n}\in \Gamma(Z,\tilde\cL^{\otimes 2n})$. Thus,  there exists $n\in \N$ such that all the sections
 $a_j\otimes (s^n_{|Z_s})\in \Gamma(Z_s,\tilde\cL^{\otimes 2n})$ extend to global sections
 $\eta_j=\widetilde{a_j s^n}\in \Gamma(Z,\tilde\cL^{\otimes 2n})$. We include $\eta_0=s^n$ in the list of such $\eta_j$. Then $\Omega(\cL, \{\eta_j\})\cap Z_s=Z_f\subset Z_s$
 follows from the equality
 \begin{equation*}
    \sum_j \eta_j\eta_j^*=\left(1+\sum_j a_j a_j^*\right) s^{2n}\in \Gamma(Z_s,\tilde\cL^{\otimes 4n})
 \end{equation*}
and since $\eta_0=s^n$ one has
$
   \cup_j Z_{\eta_j}\supset Z_s\in S(X)
$
and using \eqref{induct}  $Z_f\in S(X)$.

$(iii)$~Let $B=\cO({}_*X)$. The Gelfand transform
\begin{equation*}
    B\ni a\mapsto \gamma(a)=\hat a, \ \ \hat a(\chi)=\chi(a)\qqq \chi \in \Hom_\C(B,\C)
\end{equation*}
defines a canonical algebra homomorphism $\gamma:B\to C^\infty(X_\sss,\C)=\cO(X_\sss)$. By  the proof of $(ii)$, the elements of $\Hom_\C(B,\C)$ fulfill \eqref{selfa}, and this shows that the homomorphism $\gamma$ is compatible with the involutions. It remains to show that the morphism of schemes
$p_X\circ\gamma^*: \Sp( \cO(X_\sss))\to X$ extends the morphism of locally ringed space $\mu:X_\sss\to X_\aaa$. The underlying topological space of the locally ringed space $X_\aaa$ is the set of complex points $X(\C)$ of the algebraic variety endowed with the Zariski topology. The sheaf $\cO=\cO_X$ associates to each Zariski open set $U$ the ring $\cO(U)$ of regular functions from $U$ to $\C$. The weakest topology making these functions continuous is the ordinary topology and the map $\mu$ is the identity on points and is the inclusion $\cO(U)\subset \cO_{X_\sss}(U)$ at the level of the structure sheaves.    The locally ringed space $X_\aaa$ is the set of closed points of the algebraic scheme $X$ with the induced topology and the restriction of the structure sheaf. Note that the map which to an open set of $X$ associates its intersection with $X_\aaa$ gives an isomorphism on the categories of open sets so that the sheaves on $X$ and on $X_\aaa$ are in one to one correspondence. This simple relation fails when comparing $X_\sss$ with $\Spec(\cO(X_\sss))$. Let $\iota:X_\sss\to \Spec(\cO(X_\sss))$ be the map which associates to $x\in X_\sss$ the kernel of the character $\cO(X_\sss)\ni f\mapsto f(x)$. It is a bijection of $X_\sss$ with the set of closed points of $\Spec(\cO(X_\sss))$. It is continuous since the open sets as $D(f)$ form a basis of the topology of $\Spec(\cO(X_\sss))$, and $\iota^{-1}(D(f))$ is the open subset $\{x\in X_\sss\mid f(x)\neq 0\}\subset X_\sss$. Moreover $\iota$ defines a morphism of locally ringed spaces since for every open set $W\subset \Spec(\cO(X_\sss))$ the elements of $\cO(W)$ give, by composition with $\iota$, elements of $C^\infty(W,\C)=\Gamma(W,\cO_{X_\sss})$. One then easily checks that the composition
\begin{equation*}
    X_\sss\stackrel{\iota}{\to} \Spec(\cO(X_\sss))
    \stackrel{\gamma^*}{\to} ({}_*X)_\C \stackrel{p_X}{\to} X
\end{equation*}
maps to the closed points of $X$ and is equal to the morphism $\mu$.
 \end{proof}
We end this appendix with the following result that shows how to recover the locally ringed space $X_\sss$ from the affine scheme $\Spec(\cO(X_\sss))$.
\begin{prop} \label{scheme} Let $V$ be a smooth and compact manifold, $V_\sss$  the locally ringed space such that $\cO_{V_\sss}(U)=C^\infty(U,\C)$ for every open subset $U\subset V$ and $V_\aff=\Spec(C^\infty(V,\C))$ the affine scheme over $\C$ associated to its algebra of smooth functions.

$(i)$~The map $\iota$ which to $x\in V$ associates the kernel of the character $\cO(V_\sss)\ni f\mapsto f(x)$ is a bijection of $V$ with the maximal ideals of the $\C$-algebra $\cO(V_\sss)$.

$(ii)$~For each point $\ffp\in V_\aff$ there exists a unique maximal ideal $m(\ffp)\in V$ containing the prime ideal $\ffp$.

$(iii)$~The map $m:V_\aff=\Spec(\cO(V_\sss))\to V$, $\ffp\mapsto m(\ffp)$ is continuous and the direct image sheaf $m_*(\cO_{V_\aff})$ of the structure sheaf of $V_\aff$ is the sheaf $\cO_{V_\sss}$.

$(iv)$~The morphism of locally ringed spaces $m\circ \iota: V\to V$ is the identity.
\end{prop}
\proof One checks $(ii)$ by using a partition of unity argument. Moreover $(ii)$  implies $(i)$.

$(iii)$~We show that $m$ is continuous. Let $U\subset V$ be an open set and
 let $\tilde U=m^{-1}(U)$. Each point $\ffp\in \tilde U$ is a prime ideal contained in a maximal ideal $\ffm=x\in U$. Let then $f\in \cO(V_\sss)$ be a smooth function on $V$ with support contained in $U$ and such that $f(x)\neq 0$. For any prime ideal $\ffq\in D(f)\subset V_\aff$ one has $m(\ffq)\in {\rm Support}(f)$. Indeed, if that is not the case the germ of $f$ around $m(\ffq)$ is zero and thus $f\in \ffq$. This shows that $\tilde U=m^{-1}(U)$ is open in $\Spec(\cO(V_\sss))$.  The elements  $s\in\cO_{V_\aff}(\tilde U)$ are obtained from a covering of $\tilde U=m^{-1}(U)$ by open subsets $D(f_j)
\subset m^{-1}(U)$ and for each $D(f_j)$ an element $s_j$ of the ring $\cO(V_\sss)_{f_j}$ of fractions
with denominator a power of $f_j$, with the condition that these elements agree on the pairwise intersections. Thus to such an $s\in\cO_{V_\aff}(\tilde U)$ one can associate a function $s|_U$ by restriction to the subset of maximal ideals, and this function is smooth. The obtained map ${\rm res}:\cO_{V_\aff}(\tilde U)\to C^\infty(U,\C)=\cO_{V_\sss}(U)$ is surjective. Indeed, one can cover $U$ by open sets $U_j=\{x\in U\mid f_j(x)\neq 0\}$ where each $f_j\in C^\infty(U)$ has compact support in $U$, thus  any $h\in C^\infty(U,\C)$ is obtained from the $s_j=hf_j/f_j\in \cO(V_\sss)_{f_j}$. The map ${\rm res}:\cO_{V_\aff}(\tilde U)\to C^\infty(U,\C)$ is injective since any element of $\cO_{V_\aff}(D(f))=\cO(V_\sss)_f$ is determined by its restriction to maximal ideals.

$(iv)$~One just needs to check that $m\circ \iota$ is the identity at the level of structure sheaves and this follows from the proof of $(iii)$.
\endproof

 \begin{rem}{\rm $(i)$~As soon as the dimension of $X$ is non zero, the image of the morphism of schemes
 $\pi_X:\Spec(\cO(X_\sss))
     \to X$
     contains a non-closed point. This shows in particular that this morphism does not factorize through the morphism of locally ringed spaces $X_\sss\to X_\aaa$.

     $(ii)$~Theorem \ref{affscheme} $(iii)$, shows that the associated map of periodic cyclic homology $p_X^*: HP_*(X)\to HP_*(({}_*X)_\C)$ is injective. It is in general far from surjective.
}\end{rem}

\end{document}